\documentclass[11pt,a4paper]{article}
\usepackage{amsthm}
\usepackage{authblk}
\usepackage{amssymb,latexsym,amsmath,graphics}
\usepackage{graphicx}
\usepackage{caption}
\usepackage{subcaption}
\usepackage{enumerate}  
\usepackage{enumitem}
\usepackage{color}
\usepackage{dsfont}
\usepackage{amsmath}
\usepackage{geometry}
\usepackage{marginnote}
\usepackage{cite}
\usepackage{bbm}
\usepackage{mathrsfs}
\usepackage[titles]{tocloft}
\usepackage{hyperref} 
\hypersetup{
    pdfpagemode={UseOutlines},
    bookmarksopen,
    pdfstartview={FitH},
    colorlinks,
    linkcolor={blue}, 
    citecolor={blue}, 
    urlcolor={blue} 
}

\newcommand{\E}{{\mathbb E}}

\newcommand{\Z}{{\mathbb Z}}
\newcommand{\R}{{\mathbb R}}
\newcommand{\Cc}{\mathcal{C}}
\newcommand{\F}{\mathcal{F}}
\newcommand{\Nn}{\mathcal{N}}
\newcommand{\Hh}{\mathcal{H}}
\newcommand{\G}{\mathscr{G}}
\newcommand{\g}{{\, |\,}}
\newcommand{\Oo}{\mathcal{O}}

\newtheorem {theorem}{Theorem}

\newtheorem {definition}[theorem]{Definition}
\newtheorem {lemma}[theorem]{Lemma}
\newtheorem {proposition}[theorem]{Proposition}

\newtheorem{condition}[theorem]{Condition}

\def\tilde{\widetilde}
\def\hat{\widehat}

\def\N{\mathbb{N}}

\def\R{\mathbb{R}}

\def\1{\mathbbm{1}}

\usepackage{layout}
\title{Consistency of Bayesian inference with Gaussian process priors for a parabolic inverse problem}
\author[]{Hanne Kekkonen \thanks{Date: \today,\ Email: h.n.kekkonen@tudelft.nl}}
\affil[]{Delft Institute of Applied Mathematics, TU Delft, The Netherlands}
\date{}
\begin{document}
\maketitle

\begin{abstract}
We consider the statistical nonlinear inverse problem of recovering the absorption term $f>0$ in the heat equation 
\begin{align*}
\begin{cases}
\partial_tu-\frac{1}{2}\Delta u+fu=0 \quad &\text{on $\Oo\times(0,\textbf{T})$}\\
u = g & \text{on $\partial\Oo\times(0,\textbf{T})$}\\
u(\cdot,0)=u_0 & \text{on $\Oo$}, 
\end{cases} 
\end{align*}
where $\Oo\in\R^d$ is a bounded domain, $\textbf{T}<\infty$ is a fixed time, and $g,u_0$ are given sufficiently smooth functions describing boundary and initial values respectively. The data consists of $N$ discrete noisy point evaluations of the solution $u_f$ on $\Oo\times(0,\textbf{T})$. We study the statistical performance of Bayesian nonparametric procedures based on a large class of Gaussian process priors. We show that, as the number of measurements increases, the resulting posterior distributions concentrate around the true parameter generating the data, and derive a convergence rate for the reconstruction error of the associated posterior means. We also consider the optimality of the contraction rates and prove a lower bound for the minimax convergence rate for inferring $f$ from the data, and show that optimal rates can be achieved with truncated  Gaussian priors. 

\end{abstract}

  
\tableofcontents

\section{Introduction}\label{Sec:Introduction}   
  
Inverse problems arise from the need to extract information from indirect and noisy measurements. In many scientific disciplines, such as imaging, medicine, material sciences and engineering, the relationship between the quantity of interest and the collected data is determined by the physics of the underlying system and can be modelled mathematically. In general, we are interested in recovering some function $f$ from measurements of $G(f)$, where $G$ is the forward operator of some partial differential equation (PDE). In practice, a statistical observation scheme provides us data  
\begin{align}\label{eq:Measuremet}
Y_i=G(f)(Z_i)+\sigma W_i, \quad i=1,\cdots,N,  
\end{align}
where the $Z_i$'s are points at which the PDE solution $G(f)$ is measured, and the
$W_i$'s are standard Gaussian noise variables scaled by a fixed noise level $\sigma>0$. The inverse problem then consists of reconstructing $f$ from the noisy measurements $(Y_i,Z_i)_{i=1}^N$. Many, possibly nonlinear, inverse problems fit into this framework, including electrical impedance tomography \cite{Calderon1980, Isaacson2004}, photoacoustic tomography and several other hybrid imaging problems \cite{Bal2010, Bal2012, Kuchment2015}, and inverse scattering \cite{Colton1998, Hohage2015}. Even though inverse problems have been studied in great detail, see e.g. \cite{Benning2018, Engl1996a, Kaltenbacher2008}, statistical noise models as the one above have been analysed only more recently \cite{Arridge2019, Bissantz2007, Kaipio2004a}.  

In many applications the forward operator $G$ arising from the related PDE is non-linear in $f$, and so the negative log-likelihood function arising from the measurement \eqref{eq:Measuremet} can be non-convex. This means that many commonly used methods like Tikhonov regularisation and maximum  a priori (MAP) estimation, where one has to minimise a penalised log-likelihood function, can not  be reliably computed by standard convex optimisation techniques. There are some iterative optimisation methods, such as Landweber iteration and Levenberg-Marquardt   regularisation, that circumvent the problems arising from non-convexity \cite{Bachmayr2009, Benning2017, Kaltenbacher2008} but the Bayesian approach offers an attractive alternative for solving complex inverse problems, see e.g. \cite{Dashti2017, Stuart2010}. In the standard Bayesian approach one assigns a Gaussian prior $\Pi$ to $f$, which is then updated, given data $(Y_i,Z_i)_{i=1}^N$, into a posterior distribution for $f$, using Bayes' theorem. The posterior distribution can be used to calculate point estimates but it also delivers an estimate of the statistical uncertainty in the reconstruction. If the forward map can be evaluated numerically one can deploy modern MCMC methods, such as stochastic gradient MCMC and parallel tempering, to construct computationally efficient Bayesian algorithms even for complicated non-linear inverse problems \cite{Beskos2017, Dosso2012, Ma2015, Latz2020}, hence avoiding optimisation algorithms and inversion of $G$. 
Computational guarantees for the mixing times of such algorithms are also available  even in general high-dimensional non-linear settings \cite{Hairer2014, Nickl2020b}.

Since there is no objective way to select a prior distribution it is natural to ask how the choice of the prior affects the solution, and especially if the conclusions are asymptotically independent of the prior. Another important question that arises is whether Bayesian inference provides a statistically optimal estimate of the unknown quantity $f$. If we assume that the data are generated from a fixed 'true' function $f=f_0$, we would like to know whether the posterior mean $\overline{f}=\E^\Pi(f\g (Y_i,Z_i)_{i=1}^N)$ converges towards the ground truth,  and at what speed the posterior contracts around $f_0$. 
Nonparametric Bayesian inverse problems have been extensively studied in linear settings and the statistical validity of Bayesian inversion methods is quite well understood, see e.g. \cite{Agapiou2013, Giordano2020a, Kekkonen2016, Knapik2011, Knapik2018, Monard2019, Ray2013, Szabo2015}.  

However, non-linear inverse problems are fundamentally more challenging and very little is known about the frequentist performance of Bayesian methods. Nonparametric Bayesian posterior contraction rates for discretely observed scalar diffusions were considered in \cite{Nickl2017}, and Bernstein-von Mises theorems for time-independent Schr\"odinger equation and compound Poisson processes were studied in \cite{Nickl2020a} and \cite{Nickl2019} respectively. The frequentist consistency of Bayesian inversion in the elliptic PDE in divergence form was examined in \cite{Vollmer2013}. All the above papers employ `uniform wavelet type priors' with bounded $C^\beta$-norms. 
The consistency of the Bayesian approach, with Gaussian process
priors, in the nonlinear inverse problem of reconstructing the diffusion coefficient from noisy observations of the solution to an elliptic PDE in divergence form, has been studied in \cite{Giordano2020b}.   
Notably, building on the ideas from \cite{Monard2021} where the consistency of Baysian inversion of noisy non-Abelian X-Ray ransform is considered, \cite{Giordano2020b} also provides contraction results for general non-linear inverse problems that fulfil certain Lipschitz and stability conditions.  A general class of non-linear inverse regression models, satisfying particular analytic conditions on the model including invertibility of the related Fisher information operator, has been considered in the recent paper \cite{Monard2020}, where a general semi-parametric Bernstein-von Mises theorem is proved.   

Closely related to \cite{Giordano2020b} are the results achieved in \cite{Nickl2020} for MAP estimates associated to Gaussian precess priors, but since the proofs are based on variational methods they are very different from the Bayesian ones, and due to the non-convexity of the forward operator the MAP estimates considered can be  difficult to compute. We also mention the recent results on statistical Cald\'eron problem \cite{Abraham2020}, where a logarithmic contraction speed is proved for the problem of recovering an unknown conductivity function from noisy measurements of the voltage to current map, also known as the Dirichlet-to-Neumann map, at the boundary of the medium.

In this paper we consider the problem of recovering a coefficient of a parabolic partial differential operator from observations of a solution to the associated PDE, under given boundary and initial value conditions, corrupted by additive Gaussian noise. More precisely, we will study the heat equation with an additional absorption or cooling term that presents all the conceptual difficulties of a time dependent parabolic PDEs but allows a clean exposition; Let $g$ and $u_0$ be sufficiently smooth boundary and initial value functions respectively, and let $f:\Oo\to\R$ be an unknown absorption term determining the solutions $u_f$ of the PDE  

\begin{align}\label{eq:SchrodingerHomogeneous}
\begin{cases}
L_f(u)=\partial_tu-\frac{1}{2}\Delta_x u+fu=0 \quad &\text{on $\Oo\times(0,\textbf{T})$}\\
u = g & \text{on $\partial\Oo\times(0,\textbf{T})$}\\
u(\cdot,0)=u_0 & \text{on $\Oo$,}
\end{cases} 
\end{align}
where $\Delta_x=\sum_{i=1}^d\partial^2/\partial x_i^2$ denotes the standard Laplace operator and $\partial_t$ the time derivative. Under mild regularity conditions on $f$, assuming that $f>0$, and $g,u_0$ satisfying natural consistency conditions on  $\partial\Oo\times\{0\}$ the theory of parabolic PDEs implies that \eqref{eq:SchrodingerHomogeneous} has a unique classical solution $G(f)=u_f\in C(\overline{Q})\cap C^{2,1}(Q)$. 
The above type reaction-diffusion equations can also be used to describe ecological systems like population dynamics, with $u$ being density of prey and $f$ describing resources or the effect of predators \cite{Wu2012}, evolution of competing languages\cite{Patriarca2004, Sole2010}, and many other spread phenomena. 
Another attractive way of modelling time evolution is using stochastic PDEs, see e.g.\cite{Altmeyer2019, Hildebrandt2019} and the references therein. However, these models are usually not feasible for Bayesian analysis due to complex likelihood functions.

We show that the posterior means arising from a large class of Gaussian process priors for $f$ provide statistically consistent recovery of the unknown function $f$ in \eqref{eq:SchrodingerHomogeneous} given data \eqref{eq:Measuremet}, and provide explicit polynomial convergence rates as the number of measurements increases.  We start with contraction results for posterior distributions arising from a wide class of rescaled Gaussian process priors, similar to those considered in \cite{Giordano2020b, Monard2021}, which address the need for additional a priori regularisation of the posterior distribution to overcome the effects of non-linearity of the forward map $G$. Building on the ideas from \cite{Monard2021} and further developed in \cite{Giordano2020b}, we first show that the the posterior distributions arising from these priors optimally solve the PDE-constrained regression problem of inferring $G(f)$ from the data \eqref{eq:Measuremet}. These results can then be combined with suitable stability estimates for the inverse of $G$ to show that the posterior distribution contracts around the true parameter $f_0$, that generated the data, at certain polynomial rate, and that the posterior mean converges to the truth with the same rate. We also consider the optimality of the contraction rates and prove lower bound for the minimax convergence rate for inferring $f_0$ from the data. We note that, while the rates achieved in the first part approach the optimal rate for very smooth models, they are not in general optimal. In the second part of the paper we show that optimal rates can be achieved with truncated and rescaled Gaussian priors. 

This paper is organised as follows. The basic setting of the statistical inverse problem and the notations used in the paper can be found in Section \ref{Sec:SetUp}. The main results are stated in Section \ref{Sec:Results} and their proofs are given in Section \ref{Sec:Proofs}.

\section{A statistical inverse problem for parabolic PDEs }\label{Sec:SetUp}

\subsection{Parabolic H\"older and Sobolev spaces}
Throughout this paper $\Oo\in \R^d$, $d\in\N$, is a given non-empty, open and bounded set with smooth boundary $\partial\Oo$ and closure $\overline{\Oo}$. We define the space-time cylinder $Q=\Oo\times(0,\textbf{T})$, with $\textbf{T}\in(0,\infty)$, and denote its lateral boundary $\partial\Oo\times(0,\textbf{T})$ by $\Sigma$. 

The spaces of continuous functions defined on $\Oo\subset\R^d$ and $\overline{\Oo}$ are denoted by $C(\Oo)$ and $C(\overline{\Oo})$  respectively, and endowed with the supremum norm $\|\cdot\|_\infty$. For positive integers $k\in\N$, $C^k(\Oo)$ is the space of $k$-times differentiable functions with uniformly continuous derivatives. For non-integer $s>0$ we define $C^s(\Oo)$ as
\begin{align*}
C^{s}(\mathcal{O})
=\left\{f \in C^{\lfloor s\rfloor}(\mathcal{O})
: \ \ \forall\  |\alpha|=\lfloor s\rfloor, 
\sup _{x, y \in \mathcal{O}, x \neq y} \frac{\left|D^{\alpha} f(x)-D^{\alpha} f(y)\right|}{|x-y|^{s-\lfloor s\rfloor}}<\infty\right\},
\end{align*}
where $\lfloor s\rfloor$ denotes the largest integer less than or equal to $s$, and for any multi-index $\alpha=(\alpha_1,\dots,\alpha_d)$, $D^\alpha$ is the $\alpha$th partial differential operator. The H\"older space $C^s(\Oo)$ is normed by
\begin{align*}
\|f\|_{C^{s}(\mathcal{O})}=\sum_{|\alpha| \leq\lfloor s\rfloor}\|D^{\alpha} f(x)\|_\infty+\sum_{|\alpha|=\lfloor s\rfloor} \sup _{x, y \in \mathcal{O}, x \neq y} \frac{\left|D^{\alpha} f(x)-D^{\alpha} f(y)\right|}{|x-y|^{s-\lfloor s\rfloor}},
\end{align*}
where the second sum is removed if $s$ is an integer. We denote by $C^\infty(\Oo)=\bigcap_s C^s(\Oo)$ the set of smooth functions. We also need H\"older-Zygmund spaces $\Cc^s(\Oo)$ which can be defined as a special case of Besov spaces by $\Cc^s(\Oo)=B^s_{\infty,\infty}(\Oo)$, $s\geq0$, see \cite[Section 3.4.2]{Triebel1983} for definitions. If $s\not\in\N$ then $\Cc^s(\Oo)=C^s(\Oo)$ and we have the continuous embeddings $\Cc^{s'}(\Oo)\subsetneq C^s(\Oo)\subsetneq\Cc^s(\Oo)$, for $s\in\N\cup\{0\}$, $s'>s$.  

The classical space to look for a solution to \eqref{eq:SchrodingerHomogeneous} is the parabolic H\"older space $C^{2,1}(Q)$ defined by
\begin{align*}
C^{2,1}(Q)=\left\{f \in C(Q): \exists\ \partial_{t} f, D_{i j} f \in C(Q), i, j=1, \ldots, d\right\}.
\end{align*}
Let $\theta\in(0,1]$ and define 
\begin{align*}
\|f\|_{C^{\theta,\theta/2}(Q)}=\|f\|_\infty+[f]_{\theta, \theta/ 2}, 
\quad [f]_{\theta, \theta/ 2}=\sup _{ z_1,z_2\in Q,z_{1} \neq z_{2}} \frac{\left|f\left(z_{1}\right)-f\left(z_{2}\right)\right|}{\rho\left(z_{1}, z_{2}\right)^{\theta}},
\end{align*}
where $\rho\left(z_{1}, z_{2}\right) $ is the parabolic distance between points $z_1=(x_1,t_1)\in\R^{d+1}$ and $z_2=(x_2,t_2)\in\R^{d+1}$ given by $\rho\left(z_{1}, z_{2}\right)=(\|x_{1}-x_{2}\|_2^2+|t_{1}-t_{2}|)^{1 / 2}$.
We denote by $C^{\theta,\theta/2}(Q)$ the space of all functions $f$ for which $\|f\|_{C^{\theta,\theta/2}}<\infty$. Finally, the parabolic H\"older space $C^{2+\theta,1+\theta/2}(Q)$, $\theta\in(0,1]$ is defined as the space of all functions $f$ for which 
\begin{align*}
\|f\|_{C^{2+\theta,1+\theta/2}(Q)} 
= \sum_{|\alpha|\leq 2} \|D^\alpha f(x,t)\|_\infty
+\|\partial_t f\|_\infty
+[f]_{2+\theta,1+\theta/2}<\infty,
\end{align*}
where
\begin{align*}
[f]_{2+\theta,1+\theta/2}
=\sum_{i, j=1}^{d}[D_{ij}f]_{\theta, \theta/ 2}
+ [\partial_t f]_{\theta, \theta/ 2}.
\end{align*}
Parabolic H\"older spaces  are Banach spaces. For further details see e.g. \cite{Amann2010, Friedman1964, Krylov1996}. 
Higher order parabolic H\"older spaces can be defined in a similar way. 
We will also need parabolic H\"older--Zygmund (Besov--H\"older) spaces $\Cc^{2+\theta,1+\theta/2}(Q)=B^{2+\theta,1+\theta/2}_{\infty,\infty}(Q)$, which possess similar properties to the isotropic H\"older-Zygmund spaces, see e.g. \cite[Chapter 7.2]{Amann1995}.

Denote by $L^2(\Oo)$ the Hilbert space of square integrable functions on $\Oo$, equipped with its usual inner product $\langle\cdot, \cdot\rangle_{L^{2}(\Oo)}$. For an integer $k\geq0$, the order-$k$ Sobolev space on $\Oo$ is the separable Hilbert space
\begin{align*}
H^{k}(\mathcal{O})=\left\{f \in L^{2}(\mathcal{O}): \forall|\alpha| \leq k, \exists D^{\alpha} f \in L^{2}(\mathcal{O})\right\},\\
\end{align*}
with the inner product $\langle f, g\rangle_{H^{k}(\mathcal{O})}=\sum_{|\alpha| \leq k}\left\langle D^{\alpha} f, D^{\alpha} g\right\rangle_{L^{2}(\mathcal{O})}$. 
For a non-integer $s\geq0$, $H^s(\Oo)$ can be defined by interpolation, see, e.g. \cite[Section 1.9.1]{Lions1972}. 

We will also use parabolic Sobolev spaces $H^{s,s/2}(Q)$, with $s\geq0$, defined by
\begin{align*}
H^{s,s/2}(Q) = L^2\big((0,\textbf{T}); H^{s}(\Oo)\big)\cap
H^{s/2}\big((0,\textbf{T}); L^2(\Oo)\big),
\end{align*}
which is a Hilbert space with a norm 
\begin{align*}
\|u\|_{H^{s,s/2}(Q)}^2=\int_0^\textbf{T}\|u(\cdot,t)\|_{H^{s}(\Oo)}^2dt+\|u\|^2_{H^{s/2}((0,\textbf{T}); L^2(\Oo))},
\end{align*}
see \cite[Section 4.2.1]{Lions1972b}.   

For $s>d/2$ the Sobolev embedding theorem implies that $H^s(\Oo)$ embeds continuously into $C^r(\Oo)$ for any $s>r+d/2\geq d/2$. Let $g\in H^{s,s/2}(Q)$ and $f\in H^{s}(\Oo)$. We then have 
\begin{align}\label{eq:SobolevBound1}
\|fg\|_{H^{s,s/2}(Q)}\leq c\|f\|_{H^{s}(\Oo)}\|g\|_{H^{s,s/2}(Q)}, \quad s>d/2.
\end{align}
We will also use the following bound for $g\in H^{s,s/2}(Q)$ and $f\in \mathcal{C}^{s,s/2}(Q)$
\begin{align}\label{eq:SobolevBound2}
\|fg\|_{H^{s,s/2}(Q)}\leq c\|f\|_{\mathcal{C}^{s,s/2}(Q)}\|g\|_{H^{s,s/2}(Q)}, \quad s\geq 0.
\end{align}
The above bounds follow from similar results in isotropic spaces (see e.g. \cite{Triebel1983}).

Whenever there is no risk of confusion, we will omit the reference to the underlying domain $\Oo$ or $Q$. Attaching a subscript $c$ to any space $X$ denotes the subspace $(X_c,\|\cdot\|_{X})$ consisting of functions with compact support in $\Oo$ (or $Q$). Also, if $K$ is a non-empty compact subset of $\Oo$, we denote by $H^s_K(\Oo)$ the closed subspace of functions in $H^s(\Oo)$ with support contained in $K$. The above definitions extends without difficulty to the case where $Q$ is
replaced by its lateral boundary $\Sigma$.

We use the symbols $\lesssim$ and $\gtrsim$ for inequalities holding up to a universal
constant. For a sequence of random variables $F_N$  
we write $F_N=O_{P}(a_N)$ if for all $\varepsilon>0$ there exists $M_\varepsilon<\infty$ such that $P(|F_N|\geq M_\varepsilon a_N)<\varepsilon$ for all $N$ large enough. Finally, we denote by $\mathcal{L(F)}$ the law of a random variable $F$.

\subsection{The measurement model and Bayesian approach}\label{SubSec:BayesApp}

Let the source function $g \in C^{2+\theta,1+\theta/2}(\overline{\Sigma})$ and the initial value function $u_0\in C^{2+\theta}(\overline{\Oo})$, with $\theta\in(0,1]$, equal to zero in some neighbourhood of $\partial\Oo\times\{0\}$.
Also, let $f\in C^\theta(\Oo)$ and $f\geq f_{min}>0$. Then the initial boundary value problem  \eqref{eq:SchrodingerHomogeneous} has a unique classical solution $u_f\in C(\overline{Q})\cap C^{2+\theta,1+\theta/2}(Q)$, see e.g. \cite{Lunardi2012,Krylov1996}, and a representation in terms of Feynman-Kac formula 
\begin{align}\label{eq:FeynmanKacHomogeneous}
\begin{split}
u(x,t)= & \E^x\left(u_0(X_t)\1_{\{\tau_t=t\}}\exp\left(-\int_0^tf(X_s)ds\right)\right)\\
& +\E^x\left(g(X_{\tau_\Oo},\tau_\Oo)\1_{\{\tau_t<t\}}\exp\left(-\int_0^{\tau_\Oo} f(X_s)ds\right)\right) \quad (x,t)\in Q.
\end{split}
\end{align}
Above $\1_A$ is the idicator function of a subset $A$, $(X_s:s\geq0)$ is a $d$-dimensional Brownian motion started at $x\in\Oo$, with the exit time $\tau_\Oo$ satisfying $\sup_{x\in\Oo}\E^x(\tau_\Oo)<\infty$, and $\tau_t=\min\{\tau_\Oo,t\}$, see e.g. \cite{Freidlin1985}.
The related inverse problem is to recover $f$ given $u_f$ (and $g,u_0$). If we additionally assume that $f$ is bounded, $g\geq g_{\min}>0$ and $u_0\geq u_{0,\min}>0$ we see, using \eqref{eq:FeynmanKacHomogeneous} and Jensen's inequality, that $u_f>0$. Hence,  given $u_f$ we can simply write $f=\frac{(\frac{1}{2}\Delta_x-\partial_t)u_f}{u_f}$. The more practical question we are interested in, is how to optimally solve the above non-linear inverse problem when the observations are corrupted by statistical noise.

We consider the following parameter space for $f$: For an integer $\alpha >2+d/2$, $f_{min}>0$, and $n=n(x)$ being the outward pointing normal at $x\in\partial\Oo$, let 
\begin{align}\label{eq:ParameterSpace}
\F_{\alpha,f_{min}}=
\left\{f \in H^{\alpha}(\mathcal{O}): \inf _{x \in \mathcal{O}} f(x)>f_{\min }, \
f_{\mid \partial \mathcal{O}}=1, \
\frac{\partial^{j} f}{\partial n^{j}}_{\mid \partial \mathcal{O}}\!\!=0 \text { for } 1 \leq j \leq \alpha -1\right\}.
\end{align}

Let $f\in\F_{\alpha,f_{\min}}$ and denote by $G(f)$ the solution to \eqref{eq:SchrodingerHomogeneous}. The measurement model we consider is
\begin{align}\label{eq:Observationf}
Y_i=G(f)(Z_i)+\sigma W_i, \quad W_i\sim\Nn(0,1)\quad  i=1,\dots,N,
\end{align}
where the noise amplitude $\sigma$ is considered to be a known constant, and the design points $Z_i=(X_i,T_i)$ are drawn from a uniform distribution on $Q$. That is, for $N\in\N$,
\begin{align*}
(Z_i)_{i=1}^N\sim\mu, \quad \mu=dz/\text{vol}(Q), 
\end{align*}
with $dz$ being the Lebesgue measure and $\text{vol}(Q)=\int_Q d(x,t)$. We assume that both space and time variable follow uniform distribution to unify the the following approach. The results could also be developed for deterministic time design at the expense of introducing further technicalities.   

We will take the Bayesian approach to the inverse problem of inferring $f$ from the noisy measuremets $(Y_i,Z_i)_{i=1}^N$, and place a priori measure on the unknown parameter $f$.  Gaussian process priors are a natural choice but they are supported in linear spaces (e.g., $H^\alpha_c(\Oo)$), which is why we next define a convenient bijective re-parametrisation for $f\in\F_{\alpha,f_{\min}}$. We follow the approach of using regular link functions as in \cite{Giordano2020b,Nickl2020}.  
\begin{definition}\label{Def:LinkFunction}
1. A function $\Phi$ is called a link function if it is a smooth, strictly increasing bijective map $\Phi:\R\to (f_{min},\infty)$ satisfying $\Phi(0)=1$ and $\Phi'(t)>0$ for all $t\in\R$. \newline
2. A function $\Phi:(a,b)\to\R$:, $\infty\leq a,b\leq\infty$, is called regular if all the derivatives of $\Phi$ are bounded on $\R$. 
\end{definition}

Note that given any link function $\Phi$, one can show (see \cite[Section 3.1]{Nickl2020}) that the parameter space $\F_{\alpha,f_{min}}$ in \eqref{eq:ParameterSpace} can be written as
\begin{align*}
\mathcal{F}_{\alpha, f_{m i n}}=\left\{\Phi \circ F: F \in H_{c}^{\alpha}(\mathcal{O})\right\}. 
\end{align*}
We can then consider the solution map associated to \eqref{eq:SchrodingerHomogeneous} as one defined on $H^\alpha_c(\Oo)$;
\begin{align}\label{eq:SolutioMap2}
\G: H_c^\alpha(\Oo) \to L^{2}(\Oo), \quad F \mapsto \G(F):=G(\Phi \circ F),
\end{align}
where $G(\Phi\circ F)=G(f)$ is the solution to \eqref{eq:SchrodingerHomogeneous} with $f=\Phi\circ F\in\F_{\alpha,f_{min}}$. 

Using the above re-parametrisation $f=\Phi\circ F$ with a given link function the observation scheme \eqref{eq:Observationf} can be rewritten as
\begin{align}\label{eq:DataModel2}
Y_i=\G(F)(Z_i)+\sigma W_i, \quad i=1,\cdots,N. 
\end{align}
The random vectors $(Y_i,Z_i)$ on $\R\times Q$ are then i.i.d. with laws denoted by $P^i_F$. It follows that $P^i_F$ has the Radon--Nikodym density
\begin{align}\label{eq:Density}
p_{F}(y, z):=\frac{d P_{F}^{i}}{d y \times d \mu}(y, z)=\frac{1}{\sqrt{2 \pi \sigma^{2}}}\ e^{-\frac{(y-\mathscr{G}(F)(z))^{2}}{2\sigma^2}}, \quad y \in \mathbb{R},\ z \in Q,
\end{align}
where $dy$ denotes the Lebesgue measure on $\R$. We write $P^N_F=\otimes_{i=1}^N P^i_F$ for the joint law of $(Y_i,Z_i)_{i=1}^N$ on $\R^N\times Q^N$, and $\E^i_F$, $\E^N_F$ for the expectation operators corresponding to the laws $P^i_F$, $P^N_F$ respectively. 

We model the parameter $F\in H^\alpha_c(\Oo)$ by a Borel probability measure $\Pi$ supported on the Banach space $C(\Oo)$. Since the map $(F,(y,z))\mapsto p_F(y,z)$ can be shown to be jointly measurable the posterior distribution $\Pi(\cdot\g Y^N,Z^N)$ of $F\g Y^N,Z^N$ arising from the model \eqref{eq:DataModel2} equals to 
\begin{align*}
\Pi\left(B \mid Y^N, Z^N\right)=\frac{\int_{B} e^{\ell^N(F)} d \Pi(F)}{\int_{C(\mathcal{O})} e^{\ell^N(F)} d \Pi(F)} 
\end{align*}
for any Borel set $B \subseteq C(\Oo)$. Above 
\begin{align*}
\ell^N(F)=-\frac{1}{2 \sigma^{2}} \sum_{i=1}^{N}(Y_{i}-\mathscr{G}(F)(Z_{i}))^{2}
\end{align*}
is the joint log-likelihood function up to a constant.

\section{Posterior consistency results}\label{Sec:Results}  

We consider priors that are build around a Gaussian base prior $\Pi'$ and then rescaled by $N^{-\gamma}$, with appropriate $\gamma>0$, to provide additional regularisation to combat the non-linearity of the inverse problem, as suggested in \cite{Monard2021} and further studied in \cite{Giordano2020b}, see also \cite{Nickl2020}. We first consider certain Gaussian process priors supported on $C^\beta$, $\beta\geq2$, and show that the posterior distributions arising from these priors concentrate near sufficiently regular ground truth $F_0$ (or $f_0$) assuming that the data $(Y^N,Z^N)$ is generated through model \eqref{eq:DataModel2} with $F=F_0$. We then prove a minimax lower bound for inferring $f$ from the data, and show that the optimal convergence rate can be achieved using truncated Gaussian base priors. The proofs of the theorems can be found from Section \ref{Sec:Proofs}.

In the following we are interested in recovering $F_0$ (or $f_0$) with $H^\alpha$, $\alpha>\beta+d/2\geq 2+d/2$, smoothness. For this we assume that $g\in H^{3/2+\alpha,3/4+\alpha/2}(\Sigma)$ and $u_0\in H^{1+\alpha}(\Oo)$ satisfy the following consistency condition; There exists $\psi\in H^{2+\alpha,1+\alpha/2}(Q)$ such that 
\begin{align}\label{eq:Consistency}
\begin{split}
\text{$\psi=g$ on $\Sigma$,} &  \text{$\quad \psi(x,0)=u_0$ on $\Oo$ and}\\
\quad \partial_t^k ((\partial_t-\frac{1}{2}\Delta_x+ f) & \psi)|_{t=0} =0
\text{\ \ for $0\leq k < \frac{\alpha}{2}-\frac{1}{2}$}.  
\end{split}
\end{align}
Then $u_{f}\in H^{2+\alpha,1+\alpha/2}(Q)\subset C^{2+\beta,1+\beta/2}(Q)$, \cite[Theorem 5.3]{Lions1972b}. Note that this is a general condition for the behaviour of the source and initial value functions close the boundary $\partial\Oo\times\{0\}$, and we could simply assume that $g,u_0$ equal to zero in a neighbourhood of $\partial\Oo\times\{0\}$ as in Section \ref{SubSec:BayesApp}.
The above assumption is sufficient but could be relaxed for many of the following results. We will assume it for simplicity.

\subsection{Rescaled Gaussian priors}
We refer, e.g., to \cite[Section 2]{Gine2015} for the basic definitions of Gaussian measures and their reproducing kernel Hilbert spaces (RKHS).

\begin{condition}\label{Con:Prior}
Let $\alpha>\beta+d/2$, with some $\beta\geq2$, and let $\Hh$ be a Hilbert space continuously imbedded
into $H^\alpha_c(\Oo)$. Let $\Pi'$ be a centred Gaussian Borel probability measure on the Banach space $C(\Oo)$ that is supported on a separable measurable linear subspace of $C^\beta(\Oo)$, and assume that the reproducing-kernel Hilbert space of $\Pi'$ equals to $\Hh$.
\end{condition}

As a simple example of a base prior satisfying Condition \ref{Con:Prior} we can consider Whittle-Mat\'ern process $M=\{M(x),x\in\Oo\}$ of regularity $\alpha-d/2$, see \cite[Example 25]{Giordano2020b} or \cite[Example 11.8]{Ghosal2017} for details. Assume that the measurement \eqref{eq:DataModel2} is generated from a 'true unknown' $F_0\in H^\alpha(\Oo)$ that is supported on a given compact subset $K$ of the domain $\Oo$, and fix a smooth cutoff function $\chi\in C^\infty_c(\Oo)$ such that $\chi=1$ on $K$.  Then $\Pi' = \mathcal{L}(\chi M)$ is supported on $C^{\beta'}(\Oo)$  for any $\beta'<\alpha-d/2$, and its RKHS $\Hh = \{\chi F, F\in H^\alpha(\Oo)\}$ is continuously imbedded into $H^\alpha_c(\Oo)$ and contains $H^\alpha_K(\Oo)$. The condition $F_0\in\Hh$ then equals to the standard assumption that $F_0\in H^\alpha(\Oo)$ is supported on a strict subset $K$ of $\Oo$.

In the following we consider the re-scaled prior 
\begin{align}\label{eq:Prior}
\Pi_{N}=\mathcal{L}\left(F_{N}\right), \quad F_{N}=N^{-\frac{d}{4 \alpha+8+2 d}} F',
\end{align}
where $F'\sim\Pi'$. Then $\Pi_N$ defines a centred Gaussian prior on $C(\Oo)$, and its RKHS $\Hh_N$ is given by $\Hh$ with the norm
\begin{align*}
\|F\|_{\Hh_N}=N^{\frac{d}{4 \alpha+8+2 d}} \|F\|_{\Hh}.
\end{align*}

\begin{theorem}\label{Thm:ContractionF}  
For a fixed integer $\alpha>\beta+d/2$, $\beta\geq2$, consider the Gaussian prior $\Pi_N$ in \eqref{eq:Prior} with the base prior $\Pi'$ satisfying Condition \ref{Con:Prior} with  RKHS $\Hh$. Let $\Pi_N(\cdot\g Y^N,Z^N)$ be the resulting posterior distribution arising from observations $(Y^N,Z^N)$ in \eqref{eq:DataModel2} with $F=F_0\in\Hh$, and set $\delta_N=N^{-(\alpha+2)/(2\alpha+4+d)}$.

Then for any $D>0$ there exists a sufficiently large $L>0$ such that, as $N\to\infty$, 
\begin{align}\label{eq:ForwardConvergenceF}
\Pi_{N}\left(F:\left\|\mathscr{G}(F)-\mathscr{G}\left(F_{0}\right)\right\|_{L^{2}(Q)}>L \delta_{N} \mid Y^N,Z^N\right)=O_{P_{F_{0}}^{N}}\left(e^{-D N \delta_{N}^{2}}\right)
\end{align}
and for sufficiently large $M>0$
\begin{align}\label{eq:NormBoundF}
\Pi_{N}\left(F:\|F\|_{C^{\beta}(\Oo)}>M \mid Y^N,Z^N\right)=O_{P_{F_{0}}^{N}}\left(e^{-D N \delta_{N}^{2}}\right).
\end{align}
\end{theorem}  

Next we will formulate a theorem about the  posterior contraction around $f_0$ in $L^2$-norm. For this we need the following push-forward posterior distribution 
\begin{align}\label{eq:PushForwardPrior}
\tilde{\Pi}_N(\cdot\g Y^N,Z^N) = \mathcal{L}(f), \quad f=\Phi\circ F, \quad F\sim \Pi_N(\cdot\g Y^N,Z^N).  
\end{align}

\begin{theorem}\label{Thm:Contractionf}
Let $\Pi_N(\cdot\g Y^N,Z^N)$, $\delta_N$ and $F_0$ be as in Theorem \ref{Thm:ContractionF} with an integer $\beta\geq2$, and denote $f_0=\Phi\circ F_0$. Then for any $D>0$ there exists $L>0$ large enough such that, as $N\to\infty$, 
\begin{align*}
\tilde{\Pi}_N\big(f : \|f-f_0\|_{L^2(\Oo)}>L\delta_N^{\frac{\beta}{2+\beta}} \g Y^N,Z^N\big) 
= O_{P_{f_0}^N}(e^{-DN\delta_N^2}). 
\end{align*}
\end{theorem}

We can also show that the posterior mean $\E^\Pi(F \g Y^N,Z^N)$ of $\Pi_N(\cdot \g Y^N,Z^N)$ converges to $F_0$ with speed $\delta_N^\frac{\beta}{2+\beta}$. 

\begin{theorem}\label{Thm:PosteriorMean}
Under the assumptions of Theorem \ref{Thm:Contractionf} let $\bar{F}_N=\E^\Pi(F \g Y^N,Z^N)$ be the mean of $\Pi_N(\cdot \g Y^N,Z^N)$. Then, as $N\to\infty$,  we have 
\begin{align*}
P_{F_0}^N\big(\|\bar{F}_N-F_0\|_{L^2(\Oo)}>L\delta_N^\frac{\beta}{2+\beta}\big)\to0. 
\end{align*}  
\end{theorem}

Note that, since a composition with the link function $\Phi$ is $L^2$-Lipschitz, the above result also holds for the original potential $f$, that is, we can replace $\|\bar{F}_N-F_0\|_{L^2}$ by $\|\Phi\circ \bar{F}_N-f_0\|_{L^2}$. 
The proof of Theorem \ref{Thm:PosteriorMean} follows directly from Theorems \ref{Thm:ContractionF} and \ref{Thm:Contractionf}, and the proof of \cite[Theorem 6]{Giordano2020b}.

We now give a minimax lower bound on the rate of estimation for $f$. We also note that, by modifying the proof of Theorem \ref{Thm:LowerBound}, one can show that the rate $\delta_N$ achieved in \eqref{eq:ForwardConvergenceF} for the PDE-constrained regression problem of recovering $\mathcal{G}(F_0)$ in prediction loss is minimax optimal. Notice that the following theorem gives a lower bound that holds for any estimate for $f$, not just the one studied in this paper.

\begin{theorem}\label{Thm:LowerBound}
Let $\alpha>2+d/2$ and $g,u_0$ be as in \eqref{eq:Consistency} with $\alpha$ replaced by $\alpha+d/2$. Then there exists $c>0$ such that for $\epsilon>0$ arbitrarily small, as $N\to\infty$,
\begin{align}\label{eq:Lowerbound}
\liminf_{\hat{f}_N}\sup_{f\in\tilde{\F}_{\alpha}} 
P_f^N\Big(\|\hat{f}_N-f\|_{L^2(\Oo)}>cN^{-\frac{\alpha}{2\alpha+4+d}}\Big)\geq 1-\varepsilon,
\end{align}
where  
$\tilde{\F}_{\alpha}=
\{f \in C^\alpha(\mathcal{O}): \inf _{x \in \mathcal{O}} f(x)\geq f_{\min}>0, \|f\|_{C^\alpha}\leq B\}$, with any sufficient large $B>0$, and the infimum is taken over all measurable functions $\hat{f}_N=\hat{f}(Y^N,Z^N ,g,u_0)$, where the observations $(Y^N,Z^N)$ are generated through model  \eqref{eq:Observationf}.
\end{theorem}

We note that the rate in Theorem \ref{Thm:LowerBound} equals to $\delta_N^{\alpha/(2+\alpha)}$ and hence the rates achieved in Theorems \ref{Thm:Contractionf} and \ref{Thm:PosteriorMean}  would be optimal if we could replace $\beta$ by $\alpha$. The second part of Theorem \ref{Thm:ContractionF} with the proof of Theorem \ref{Thm:Contractionf} reveals that this suboptimal rate is due to the fact that, while we are interested in recovering $F_0\in \mathcal{H}\subset H_c^\alpha$, the posterior mass is concentrated in $C^\beta$-balls, with $\beta<\alpha-d/2$. To overcome this problem we next consider truncated Gaussian priors, whose posterior mass will be shown to be concentrated in some $H^\alpha$-balls, and which attain the optimal convergence rate.

\subsection{Truncated Gaussian priors}

In practice so called sieve-priors, which are concentrated on a finite-dimensional approximation of the parameter space supporting the prior, are often employed for computational reasons. One of the commonly used methods is to use the truncated  Karhunen--Lo\`eve series expansion of the Gaussian base prior $\Pi'$. The contraction rate \eqref{eq:ForwardConvergenceF} of the forward problem remains valid with these truncated priors if the approximation spaces are appropriately chosen and we show that the optimal rate of estimating $f_0$ can be achieved.

Let $\{\Psi_{l,r}, l\geq-1, r\in\Z^d\}$ be an orthonormal basis of $L^2(\R^d)$ composed of sufficiently regular, compactly supported Daubechies wavelets, see the proof of Theorem \ref{Thm:LowerBound} for more details. We assume that $F_0\in H^\alpha_K(\Oo)$ for some $K\in\Oo$, and denote by $\mathcal{R}_l$ the set of indices $r$ such that the support of $\Psi_{l,r}$ intersects with $K$. Fix any compact $K'\subset\Oo$ such that $K\subsetneq K'$ and a cut-off function $\chi\in C^\infty_c(\Oo)$ for which $\chi=1$ on $K'$. Let $\alpha>2+d/2$, and consider the prior $\Pi'_J$ arising as the law of the Gaussian sum 
\begin{align}\label{eq:TruncatedPrior}
\Pi_J'=\mathcal{L}(\chi \tilde{F}),\quad 
\tilde{F}=\sum_{\substack{l \leq J\\ r\in\mathcal{R}_l}} 2^{-\alpha l}F_{l,r}\Psi_{l,r}, \quad
F_{l,r}\stackrel{\text { iid }}{\sim}\mathcal{N}(0,1),  
\end{align}
where $J=J_N\to\infty$ is a deterministic truncation point. Then $\Pi_J'$ defines a centred Gaussian prior that is supported on a finite dimensional space 
\begin{align*}
\mathcal{H}_J:=\text{span}\{\chi \Psi_{l,r}, l\leq J, r\in\mathcal{R}_l\}\subset C(\Oo).
\end{align*}

We will next show that a posterior mean estimate arising from the above truncated and rescaled Gaussian prior attains the optimal convergence rate \eqref{eq:Lowerbound}. 

\begin{theorem}\label{Thm:TruncatedPrior}
Let $\Pi_N$ be the rescaled prior as in \eqref{eq:Prior}, where now $F'\sim \Pi_J'$ and $J=J_N\in\N$ is chosen so that $2^J\simeq N^{1/(2\alpha+4+d)}$. Let $\Pi_N(\cdot\g Y^N,Z^N)$ be the resulting posterior distribution arising from data $(Y^N,Z^N)$ in \eqref{eq:DataModel2} with $F=F_0\in H^\alpha_K(\Oo)$. Let $\delta_N$ be as in Theorem \ref{Thm:ContractionF}, and assume that $f_0=\Phi\circ F_0$. Then \eqref{eq:ForwardConvergenceF} remains valid and for any $D>0$, and for sufficiently large $M>0$, 
\begin{align}\label{eq:PosteriorHaBall}
\Pi_{N}\left(F:\|F\|_{H^{\alpha}(\Oo)}>M \mid Y^N,Z^N\right)=O_{P_{F_{0}}^{N}}\Big(e^{-D N \delta_{N}^{2}}\Big),
\end{align}
as $N\to\infty$.
Furthermore, let $\tilde{\Pi}_N(\cdot\g Y^N,Z^N)$ be the push-forward posterior as in \eqref{eq:PushForwardPrior}. Then there exists $L>0$ large enough  such that
\begin{align}\label{eq:OptimalRate}
\tilde{\Pi}_N\big(f : \|f-f_0\|_{L^2(\Oo)}>L N^{-\frac{\alpha}{2\alpha+4+d}}\g Y^N,Z^N\big) 
= O_{P_{f_0}^N}\Big(e^{-DN\delta_N^2}\Big),
\end{align}
as $N\to\infty$.
\end{theorem}

\begin{theorem}\label{Thm:PosteriorMeanTruncated}
Under the assumptions of Theorem \ref{Thm:TruncatedPrior} let $\bar{F}_N=\E^\Pi(F \g Y^N,Z^N)$ be the mean of $\Pi_N(\cdot \g Y^N,Z^N)$. Then, as $N\to\infty$,  we have 
\begin{align*}
P_{f_0}^N\big(\|\Phi\circ \bar{F}_N-f_0\|_{L^2(\Oo)}>L N^{-\frac{\alpha}{2\alpha+4+d}}\big)\to0. 
\end{align*}  
\end{theorem}

The proof of Theorem \ref{Thm:PosteriorMeanTruncated} follows directly from Theorem \ref{Thm:TruncatedPrior} and the proof of \cite[Theorem 6]{Giordano2020b}. 

From a non-asymptotic point of view, the $N$-dependent rescaling of the prior can be thought just as an adjustment of the covariance operator of the prior. A natural question that arises is whether the non-linear inverse problem considered here can be solved in a fully Bayesian way (prior independent of the measurement). One way of achieving this would be to employ a hierarchical prior that randomises the finite truncation point $J$ in the above Karhunen-Loéve series expansion. Such an approach has been investigated in \cite{Giordano2020b} for an elliptic PDE with similar consistency results as Theorems \ref{Thm:ContractionF} -\ref{Thm:PosteriorMean} for smooth enough ground truth. We do not pursue the topic in this paper.

\section{Proofs}\label{Sec:Proofs}

\subsection{Proofs of the main results}
The proof of Theorem \ref{Thm:PosteriorMean} follows directly from Theorems \ref{Thm:ContractionF} and \ref{Thm:Contractionf}, and the proof of \cite[Theorem 6]{Giordano2020b}.
The proofs of Theorems \ref{Thm:ContractionF} and \ref{Thm:Contractionf} rely on the following forward and stability estimates. The proofs of the Propositions can be found in Section \ref{SubSec:PropProofs}.

\begin{proposition}\label{Prop:ForwardEstimate}
Let $\G$ be the solution map defined in \eqref{eq:SolutioMap2} with $g,u_0$ as in \eqref{eq:Consistency}.
Let $\alpha>2+d/2$ and $F_1,F_2\in H^\alpha_c(\Oo)$. Then 
\begin{align*}
\|\G(F_1)-\G(F_2)\|_{L^2(Q)}\lesssim(1+\|F_1\|_{C^2(\Oo)}^4\vee\|F_2\|_{C^2(\Oo)}^4)\|F_1-F_2\|_{(H^2(\Oo))^*},
\end{align*}  
where $X^*$ denotes the topological dual space of a normed linear space $X$. Also, there exists $C>0$ such that 
\begin{align*}
\sup_{F\in H^\alpha_c(\Oo)}\|\G(F)\|_\infty\leq C(\|g\|_\infty+\|u_0\|_\infty)<\infty.
\end{align*}
\end{proposition}

\begin{proof}[Proof of Theorem \ref{Thm:ContractionF}]    
It follows from Propositions \ref{Prop:ForwardEstimate} that the inverse problem \eqref{eq:DataModel2} falls in the general framework studied in \cite{Giordano2020b}, with $\beta=2$, $\gamma=4$, $\kappa=2$ and $S=c(\|g\|_\infty+\|u_0\|_\infty)$. Theorem \ref{Thm:ContractionF} then follows directly from \cite[Theorem 14]{Giordano2020b}. 
\end{proof}

\begin{proposition}\label{Prp:StabilityEstimate}  
Let $G(f)$ be the solution to \eqref{eq:SchrodingerHomogeneous} with $g,u_0$ as in \eqref{eq:Consistency} and $f,f_0\in\F_{\alpha,f_{\min}}$. Then 
\begin{align*}
\|f-f_0\|_{L^2(\Oo)}\lesssim & \
e^{c\|f\vee f_0\|_\infty}\|G(f)-G(f_0)\|_{H^{2,1}(Q)}.\\
\end{align*}  
\end{proposition}

To prove Theorem \ref{Thm:Contractionf} we will also need that the forward map $G$ maps bounded sets in $\mathcal{C}^\beta$ onto bounded sets in $H^{2+\beta,1+\beta/2}(Q)$.

\begin{proposition}\label{Prp:NormEstimate}
Let $\beta>0$ and $f\in C^\beta(\Oo)$, with $\inf_{x\in\Oo}f(x)\geq f_{\min}>0$. 
Then there exists a constant $C>0$ such that 
\begin{align*}
\|G(f)\|_{H^{2+\beta,1+\beta/2}(Q)}\leq C(1+\|f\|_{\Cc^{\beta}}^{1+\beta/2}).
\end{align*}
\end{proposition}

\begin{proof}[Proof of Theorem \ref{Thm:Contractionf}]
We start by noting that if $\Phi:\R\to\R$ is a regular link function in the sense of Definition \ref{Def:LinkFunction} then for each integer $m\geq0$ there exists $C>0$ such that, for all $F\in C^m(\Oo)$,
\begin{align}\label{eq:LinkBound}
\|\Phi \circ F\|_{C^m}\leq C(1+\|F\|_{C^m}^m).
\end{align}
See \cite[Lemma 29]{Nickl2020} for proof. With the above bound we can use the conclusion of Theorem \ref{Thm:ContractionF} for the push-forward posterior $\tilde{\Pi}_N(\cdot \g Y^N,Z^N)$. Estimate \eqref{eq:ForwardConvergenceF} directly implies that 
\begin{align*}
\tilde{\Pi}_N(f : \|G(f)-G(f_0)\|_{L^2}>L\delta_N \g Y^N,Z^N)=O_{P_{f_0}^N}(e^{-DN\delta_N^2}),
\end{align*}
as $N\to\infty$. Using the bound \eqref{eq:LinkBound} and the estimate \eqref{eq:NormBoundF} we get, for a sufficiently large $M'>0$,   
\begin{align*}
\tilde{\Pi}_N(f : \|f\|_{C^\beta}>M' \g Y^N,Z^N)
\leq \Pi_{N}\left(F:\|F\|_{C^{\beta}}>M \g Y^N,Z^N\right)
=O_{P_{F_{0}}^{N}}\left(e^{-D N \delta_{N}^{2}}\right). 
\end{align*}
Since $C^\beta \subset \Cc^\beta$ the above estimate is still true if we replace $\|\cdot\|_{C^\beta}$ by $\|\cdot\|_{\Cc^\beta}$.

If $f\in C^\beta$ with $\|f\|_{\Cc^\beta}\leq M'$ Lemma \ref{Prp:NormEstimate} implies that $G(f),G(f_0)\in H^{2+\beta,1+\beta/2}$ and 
\begin{align*}
\|G(f_0)\|_{H^{2+\beta,1+\beta/2}}\lesssim 1+\|f_0\|_{\Cc^\beta}^{1+\beta/2}<\infty, 
\|G(f)\|_{H^{2+\beta,1+\beta/2}}\lesssim 1+\|f\|_{\Cc^\beta}^{1+\beta/2}\leq M''<\infty.
\end{align*}
We will also need the following interpolation inequality.
Let $s\geq0$ and $\theta
\in(0,1)$. Then for all $u\in H^{s,s/2}(Q)$  
\begin{align}\label{eq:Interpolation}
\|u\|_{H^{(1-\theta)s,(1-\theta)s/2}}\lesssim \|u\|_{H^{s,s/2}}^{1-\theta}\|u\|_{L^2}^\theta,
\end{align}
see \cite[Chapter 4, Proposition 2.1]{Lions1972b}.
Combining the above we see that 
\begin{align*}
\|G(f)-G(f_0)\|_{H^{2,1}}
& \lesssim\|G(f)-G(f_0)\|_{H^{2+\beta,1+\beta/2}}^\frac{2}{2+\beta}
\|G(f)-G(f_0)\|_{L^2}^\frac{\beta}{2+\beta}\\
& \lesssim \|G(f)-G(f_0)\|_{L^2}^\frac{\beta}{2+\beta}
\end{align*}

Hence we get, for large enough $L>0$,
\begin{align*}
\tilde{\Pi}_N & (f : \|G(f)-G(f_0)\|_{H^{2,1}} > L\delta_N^\frac{\beta}{2+\beta} \g Y^N,Z^N)\\
& \leq \tilde{\Pi}_N(f : \|G(f)-G(f_0)\|_{L^2} > L\delta_N \g Y^N,Z^N)
+ \tilde{\Pi}_N(f : \|f\|_{\Cc^\beta}>M' \g Y^N,Z^N)\\
& =O_{P_{f_0}^N}(e^{-ND\delta_N^2}), 
\end{align*}
as $N\to\infty$.
Applying the stability estimate of Proposition \ref{Prp:StabilityEstimate} we can then conclude 
\begin{align*}
\tilde{\Pi}_N & (f : \|f-f_0\|_{L^2} > L\delta_N^\frac{\beta}{2+\beta} \g Y^N,Z^N)\\
& \leq \tilde{\Pi}_N (f : \|G(f)-G(f_0)\|_{H^{2,1}} > L'\delta_N^\frac{\beta}{2+\beta} \g Y^N,Z^N)
+ \tilde{\Pi}_N(f : \|f\|_{\Cc^\beta}>M' \g Y^N,Z^N)\\
& =O_{P_{f_0}^N}(e^{-ND\delta_N^2}). 
\end{align*}
\end{proof}

\begin{proof}[Proof of Theorem \ref{Thm:LowerBound}]
The proof follows ideas of \cite{Nickl2020a, Nickl2020} modified for the parabolic problem considered in this paper.
Consider an $s$-regular orthonormal wavelet basis for the Hilbert space $L^2(\R^d)$ given by compactly supported Daubechies tensor wavelet basis functions 
\begin{align*}
\{\Psi_{l,k}\ :\ l\in \N\cup\{-1,0\}, k\in\Z^d\}\quad \Psi_{l,k}=2^\frac{ld}{2}\Psi_{0,k}(2^l\cdot), \ \text{for $\l\geq0$.}
\end{align*}
Note that we can chose the smoothness $s$ of the basis to be as large as required and hence will omit it in what follows. For more details about wavelets see \cite{Daubechies2006, Meyer1995} or \cite[Chapter 4]{Gine2015}.
For $\alpha\in\R$ we have the following wavelet characterisation of $H^\alpha(\R^d)$ norm
\begin{align}\label{eq:HnormWavelet1}
\|f\|_{H^\alpha(\R^d)}^2\simeq \sum_{l,k}2^{2l\alpha}\langle f,\Psi_{l,k}\rangle_{L^2(\R^d)}^2.
\end{align}
We also note that, for $\alpha\geq 0$ and some $C>0$, 
\begin{align*}
f\in C^\alpha(\R^d) \Rightarrow \sup_{l,k}2^{l(\alpha+d/2)}|\langle f,\Psi_{l,k}\rangle_{L^2(\R^d)}|\leq C\|f\|_{C^\alpha(\R^d)}, 
\end{align*}
with the converse being true when $\alpha\not\in\N$. 

We can construct an orthonormal wavelet basis of $L^2(\Oo)$ given by 
\begin{align*}
\Big\{\Psi^\Oo_{l,k}\ :\ k\leq N_l,\ l\in\N\cup\{-1,0\}\Big\}, \quad N_l\in\N,
\end{align*}
that consists of interior wavelets $\Psi^\Oo_{l,k}=\Psi_{l,k}$, which are compactly supported in $\Oo$, and of boundary wavelets $\Psi^\Oo_{l,k}=\Psi^b_{l,k}$, which are an orthonormalised linear combination of those wavelets that have support inside and outside $\Oo$, see \cite[Theorem 2.33]{Triebel2008}.   
Using the above basis any function $f\in L^2(\Oo)$ has orthogonal wavelet series expansion
\begin{align*}
f=\sum_{l}\sum_{k=1}^{N_l}\langle f,\Psi_{l,k}^\Oo\rangle_{L^2(\Oo)}^2\Psi_{l,k}^\Oo. 
\end{align*}

  
We will apply \cite[Theorem 6.3.2]{Gine2015} and its proof which reduces the problem of estimating the lower bound in the whole parameter space into a testing problem in a finite subset $(f_m : m=1,\dots, M)$ of $\tilde{F}_\alpha$. To do this, we need to find a suitable lower bound between the $f_m$'s in the $L^2(\Oo)$ distance, and an appropriate upper bound in the Kullback--Leibler divergence of the laws $P_{f_m}^N$ and $P_{f_0}^N$.

1. We start by showing that $\tilde{F}_\alpha$ contains $\{f_m\ :\ m=0,1,\dots,M\}$, $M\geq1$ such that
\begin{align*} 
\|f_m-f_{m'}\|_{L^2}\gtrsim N^{-\frac{\alpha}{2\alpha+4+d}}
\quad \text{for all $m\not=m'$, }
\end{align*}
that is, $f_m$ are $N^{-\frac{\alpha}{2\alpha+4+d}}$-separated from each other. 

For every $j\in\N$ there exists a small positive constant $c$ and $n_j=c2^{jd}$ many Daubechies wavelets $(\Psi_{jr}\ :\ r=1,\dots,n_j)$ with disjoint compact supports in $\Oo$. Let $b_{m,\cdot}$ be a point in the discrete hypercube $\{-1,1\}^{n_j}$. We define 
\begin{align}\label{eq:hm}
h_m(x)=h_{m,j}(x)=\kappa\sum_{r=1}^{n_j} b_{m,r}2^{-j(\alpha+d/2)}\Psi_{j,r}(x), 
\quad x\in\Oo,
\end{align}
where $\kappa$ can be chosen to be as small as wanted, and $m=0,\dots,M_j$ with  $M_j$ chosen later. Note that $h_m$ is compactly supported in $\Oo$ and has zero extension from $\Oo$ to $\R^d$ with the global H\"older norm being equal to the intrinsic one. For $\alpha\in\N$ we can write 
\begin{align*}
\sup_{x\in\R^d}|D^\alpha h_m(x)| & =\sup_{x\in\Oo} \Big|\kappa\sum_{r=1}^{n_j} b_{m,r}2^{-j(\alpha+d/2)}D^\alpha\Psi_{j,r}(x)\Big|\\
& \leq \kappa\sup_{x\in\Oo}\sum_{r=1}^{n_j} |(D^\alpha\Psi_{0,r})(2^jx)|\leq C\kappa,
\end{align*}
where the last inequality follows from the fact that at any point there are only finitely many $\Psi_{0,r}$ that get a non-zero value. Since the interior wavelets are orthogonal to the boundary wavelets we get for $\alpha\not\in\N$ 
\begin{align*}
\|h_m\|_{C^\alpha}\leq C\sup_{l,k}2^{l(\alpha+d/2)}|\langle h_m,\Psi_{l,k}\rangle_{L^2}|=C\kappa.
\end{align*}
Hence, by choosing $\kappa$ small enough, we see that all $h_m$ are contained in $\{ h\in C^\alpha(\Oo)\ :\ \|h\|_{C^\alpha}\leq1\}$. 

Let $f_0=1$, $h_m$ as in \eqref{eq:hm}, and define functions 
\begin{align*}
f_m=f_0+h_m, \quad m=1,\dots,M_j.
\end{align*}
We then have $\|f_m\|_{C^\alpha}\leq\|f_0\|_{C^\alpha}+C\kappa$, and for $\kappa$ small enough $f_m$ is bounded away from zero. By the Varshamov--Gilbert bound there exists $\{b_{m,\cdot}:m=1,\dots,M_j\}\in\{-1,1\}^{n_j}$, with $M_j\geq3^{\frac{n_j}{4}}$, that are $n_j/8$-separated in the Hamming-distance, that is, for all $m,m'\leq M_j$ and $m\not=m'$
\begin{align*}
\sum_{r=1}^{n_j}(b_{m,r}-b_{m',r})^2\gtrsim n_j    
\end{align*}
see e.g. \cite[Example 3.1.4]{Gine2015}. Setting $2^j\simeq N^{\frac{1}{2\alpha+4+d}}$ we get, for the $f_m$ corresponding to the above separated $b_{m,\cdot}$, that
\begin{align*}
\|f_m-f_{m'}\|_{L^2}^2 & = \|h_m-h_{m'}\|_{L^2}^2\\
& = \kappa^2 2^{-2j(\alpha+d/2)}\Big\|\sum_{r=1}^{n_j}(b_{m,r}-b_{m',r})\Psi_{j,r}\Big\|_{L^2}^2\\
& = \kappa^2 2^{-2j(\alpha+d/2)}\sum_{r=1}^{n_j}(b_{m,r}-b_{m',r})^2\\
& \gtrsim \kappa^2 2^{-2j(\alpha+d/2)}n_j 
\gtrsim N^{\frac{2\alpha}{2\alpha+4+d}}. 
\end{align*}

2. Next we will prove that, for some $\epsilon>0$, 
\begin{align*}
KL(P_{f_m}^N,P_{f_0}^N)\leq \epsilon\log(M),
\end{align*}  
where $KL$ denotes the Kullback--Leibler divergence. 

Using \eqref{eq:Density} we see that  
\begin{align*}
\E_{f_m}^i\Bigg(\log\frac{dP_{f_m}^i(Y_i,Z_i)}{dP_{f_0}^i(Y_i,Z_i)}\Bigg) 
& =\E_{f_m}^i\bigg(\frac{1}{2\sigma^2}\Big((Y_i-u_{f_0}(Z_i))^2-(Y_i-u_{f_m}(Z_i))^2\Big)\bigg)\\
& =\frac{1}{2\sigma^2}\E^{\mu}\big(u_{f_0}(Z_i)^2-2u_{f_0}(Z_i)u_{f_m}(Z_i)+u_{f_m}(Z_i)^2\big)\\
& \simeq \|u_{f_0}-u_{f_m}\|^2_{L^2}. 
\end{align*}
Since $P^N_{f_m}$ is the product measure $\otimes_{i=1}^N P^i_F$ we have $KL(P^N_{f_m},P^N_{f_0})\simeq N\|u_{f_0}-u_{f_m}\|^2_{L^2}$. 
Using Lemma \ref{Lem:ForwardEstimate} and \eqref{eq:HnormWavelet1} we then get 
\begin{align*}  
\|u_{f_m}-u_{f_0}\|_{L^2(Q)}^2 & \lesssim \|f_m-f_0\|_{(H_0^2(\Oo))^*}^2 \\
& \lesssim\|h_m\|_{H^{-2}(\R^d)}^2\\
& =\kappa^2 2^{-2j(\gamma+d/2+2)}\sum_{r=1}^{n_j}1  
\lesssim \kappa^2 N^{-1}n_j.
\end{align*}
By the definition of $M_j$ and choosing $\kappa$ small enough we can conclude that $KL(P_{f_m}^{N},P_{f_0}^{N})\leq\epsilon\log(M_j)$. 

Theorem 6.3.2 from \cite{Gine2015} then states that 
\begin{align*}
N^{\frac{\alpha}{2\alpha+4+d}}\inf_{\hat{f}_N}\sup_{f\in\tilde{\F}_{\alpha}} \E_f^{N}\|\hat{f}_N-f\|_{L^2(\Oo)}
\geq \frac{\sqrt{{M_N}}}{1+\sqrt{{M_N}}}\left(1-2 \epsilon-\sqrt{\frac{8 \epsilon}{\log {M_N}}}\right), 
\end{align*}  
where $M_N=M_j\to\infty$ when $N\to\infty$, and $\epsilon$ can be chosen to be as small as wanted by choosing $\kappa$ small enough.  
We note that the proof of \cite[Theorem 6.3.2]{Gine2015} actually states a slightly stronger result, showing the above lower bound for a certain testing problem, as stated in (6.105). This, combined with the fact that $P_f^N\Big(\|\hat{f}_N-f\|_{L^2(\Oo)}>N^{-\frac{\alpha}{2\alpha+4+d}}\Big)$ can be lower bound by that testing problem, see (6.101) and (6.99) from \cite{Gine2015}, concludes the proof of Theorem \ref{Thm:LowerBound}. 
\end{proof}

\begin{proof}[Proof of Theorem \ref{Thm:TruncatedPrior}]
We denote by $F^J=P_{\mathcal{H}_J}(F)\in\mathcal{H}_J$ the projection 
\eqref{eq:projection} and note that for all $F_0\in H^\alpha_K(\Oo)$
\begin{align*}
\|F_0-F_0^J\|_{(H^2(\Oo))^*}\lesssim 2^{-J(\alpha+2)}
\end{align*}
by (63) from \cite{Giordano2020b}. We also note that $\|F_0^J\|_{C^2}\le\||F_0\|_{C^2}$ by standard properties of wavelet bases and hence, choosing $2^J=N^{1/(2\alpha+4+d)}$, it follows from Proposition \ref{Prop:ForwardEstimate} that 
\begin{align*}
\|\G(F_0)-\G(F_0^J)\|_{L^2(Q)}\lesssim \|F_0-F_0^J\|_{(H^2(\Oo))^*}
\lesssim N^{-\frac{\alpha+2}{2\alpha+4+d}}=\delta_N.
\end{align*}  
Using triangle inequality we then see that for a sufficiently large $c>0$
\begin{align*}
\Pi_N(\|\G(F)-\G(F_0)\|_{L^2(Q)}>c\delta_N)
\geq \Pi_N(\|\G(F)-\G(F_0^J)\|_{L^2(Q)}>c'\delta_N).
\end{align*}
We can then conclude that the results of Theorems \ref{Thm:ContractionF}-\ref{Thm:PosteriorMean} remain valid under the truncated and rescaled Gaussian prior, see \cite{Giordano2020b} Section 3.2.

Following the idea from the proof of Theorem 4.13 \cite{Nickl2020b}
we will next show that the posterior mass concentrates in some $H^\alpha$-balls. 
Define for any $M',Q>0$
\begin{align*}  
\mathcal{A}_N  =\{F=F_1+F_2\in\mathcal{H}_J:\|F_1\|_{(H^2)^*}\leq Q\delta_N,\ \|F_2\|_{\mathcal{H}}\leq M'\}.
\end{align*}
Then, by Theorem 13 with Lemmas 17 and 18 in \cite{Giordano2020b}, we have for $Q,M'$ sufficiently large
\begin{align*}
\Pi_N(F\in\mathcal{A}_N\g Y^N,Z^N) \geq 1-O_{P_{F_0}^N}(e^{DN\delta_N^2}).
\end{align*}
To prove \eqref{eq:PosteriorHaBall} we need to show that $\|F\|_{H^\alpha}\leq M$ for all $F\in\mathcal{A}_N$.

Let $\alpha\geq0$. Then 
\begin{align*}
\tilde{F}=\sum_{\substack{l \leq J\\ r\in\mathcal{R}_l}} F_{l,r}\Psi_{l,r}
=\sum_{\substack{l \leq J\\ r\in\mathcal{R}_l}} 2^{-\alpha l}G_{l,r}\Psi_{l,r}
, \quad G_{l,r}\stackrel{\text { iid }}{\sim}\mathcal{N}(0,1),  
\end{align*}
defines a centred Gaussian probability measure supported on $\tilde{\mathcal{H}}_J=\text{span}\{ \Psi_{l,r}, l\leq J, r\in\mathcal{R}_l\}$ with the RKHS $\tilde{\mathcal{H}}_J$ endowed with norm
\begin{align*}
\|\tilde{F}\|_{\tilde{\mathcal{H}}_J}^2
= \sum_{\substack{l \leq J\\ r\in\mathcal{R}_l}} 2^{2l\alpha}F_{l,r}^2
=\|\tilde{F}\|_{H^\alpha(\R^d)}^2\quad \forall \tilde{F}\in\tilde{\mathcal{H}}_J.
\end{align*}
The random function 
\begin{align*}
F=\chi\tilde{F}
=\sum_{\substack{l \leq J\\ r\in\mathcal{R}_l}} F_{l,r}\chi\Psi_{l,r}
=\sum_{\substack{l \leq J\\ r\in\mathcal{R}_l}} 2^{-\alpha l}G_{l,r}\chi\Psi_{l,r}
, \quad G_{l,r}\stackrel{\text { iid }}{\sim}\mathcal{N}(0,1),    
\end{align*}
then defines the centred Gaussian probability measure $\Pi_J'$, as in \eqref{eq:TruncatedPrior}, supported on $\mathcal{H}_J=\text{span}\{ \chi\Psi_{l,r}, l\leq J, r\in\mathcal{R}_l\}$, with the RKHS norm satisfying
\begin{align*}
\|F \|_{\mathcal{H}_J}^2
= \|\chi\tilde{F} \|_{\mathcal{H}_J}^2
\leq \|\tilde{F}\|_{\tilde{\mathcal{H}}_J}^2
= \sum_{\substack{l \leq J\\ r\in\mathcal{R}_l}} 2^{2l\alpha}F_{l,r}^2 
\quad \forall F\in{\mathcal{H}}_J.
\end{align*}
We also note that for all $\tilde{F}\in\tilde{\mathcal{H}}_J$ there exists $\tilde{F}'\in\tilde{\mathcal{H}}_J$ such that $\chi \tilde{F}=\chi \tilde{F}'$ and 
\begin{align*}
\|\chi\tilde{F}\|_{\mathcal{H}_J}=\|\tilde{F}'\|_{\tilde{\mathcal{H}}_J}.
\end{align*}
Hence, if $F=\chi\tilde{F}$ is an arbitrary element of $\mathcal{H}_J$ we can write 
\begin{align*}
\|F\|_{H^\alpha(\Oo)}
= \|\chi\tilde{F}'\|_{H^\alpha(\Oo)}
\leq \|\tilde{F}'\|_{H^\alpha(\R^d)}
=\|\tilde{F}'\|_{\tilde{\mathcal{H}}_J}
=\|F\|_{\mathcal{H}_J}\quad \forall F\in\mathcal{H}_J. 
\end{align*}

We will next show that
\begin{align*}
\|F\|_{H^\alpha(\Oo)}\lesssim 2^{J(\alpha+2)}\|F\|_{(H^2(\Oo))^*}\quad \forall F\in\mathcal{H}_J.
\end{align*}
Denote for $J'\in\N,\ J'\leq J$, 
\begin{align}\label{eq:projection}
F^{J'}=P_{\mathcal{H}_{J'}}(F)
=\sum_{\substack{l \leq {J'}\\ r\in\mathcal{R}_l}} \langle F,\Psi_{l,r}\rangle\chi\Psi_{l,r} 
\in \mathcal{H}_{J'}.
\end{align} 
Note that for large enough $J_{\min}\in\N$, if $l\geq J_{\min}$ and the support of $\Psi_{l,r}$ intersects $K$, then $\text{supp}(\Psi_{l,r})\subset K'$ and we have $\chi\Psi_{l,r}=\Psi_{l,r}$ for all $l\geq J_{\min}$ and $r\in\mathcal{R}_l$. We can then write 
\begin{align*}
F=F^{J_{\min}}+(F-F^{J_{\min}})
=\sum_{\substack{l \leq J_{\min}\\ r\in\mathcal{R}_l}} F_{l,r}\chi\Psi_{l,r}
+\sum_{\substack{J_{\min}<l \leq J\\ r\in\mathcal{R}_l}} F_{l,r}\Psi_{l,r}.
\end{align*}
Since $\mathcal{H}_{J_{\min}}$ is a fixed finite dimensional subspace we get, by equivalence of norms, that $\|F^{J_{\min}}\|_{H^\alpha}\leq C_{J_{\min}}\|F^{J_{\min}}\|_{(H^2)^*}\leq C_{J_{\min}}\|F\|_{(H^2)^*}$. We also see that $F-F^{J_{\min}}$ is compactly supported on $\Oo$ and hence can be extended by zero to $\R^d$. We can then write 
\begin{align*}
\|F-F^{J_{\min}}\|_{H^\alpha(\Oo)}^2
& = \sum_{\substack{J_{\min}<l \leq J\\ r\in\mathcal{R}_l}} 2^{2l\alpha}F_{l,r}^2\\
& = \sum_{\substack{J_{\min}<l \leq J\\ r\in\mathcal{R}_l}} 2^{2l\alpha+4l}2^{-4l}F_{l,r}^2\\
& \leq 2^{J(2\alpha+4)}\|F-F^{J_{\min}}\|_{(H^2(\Oo))^*}^2\\
& \leq 2^{J(2\alpha+4)}\|F\|_{(H^2(\Oo))^*}^2
\end{align*}

Combining the above we see that, for $F\in\mathcal{H}_J,\ \|F\|_{(H^2)^*}\leq Q\delta_N$ and $2^J=N^{1/(2\alpha+4+d)}$,
\begin{align*}
\|F\|_{H^\alpha(\Oo)}
\leq 2^{J(\alpha+2)}\|F\|_{(H^2(\Oo))^*}
\leq Q N^\frac{\alpha+2}{2\alpha+4+d}\delta_N=Q,
\end{align*}
which concludes the first part of the proof. 

To prove the optimal convergence rate we need to replace the $\|\cdot\|_{C^\beta}$-bounds in in the proof of Theorem \ref{Thm:Contractionf} by $\|\cdot\|_{H^\alpha}$-bounds. To do this we note that if $\alpha>d/2$ and $f\in H^\alpha(\Oo)$ with $\inf_{x\in\Oo}f(x)\geq f_{\min}>0$ we can use the inequality \eqref{eq:SobolevBound1} instead of \eqref{eq:SobolevBound2} in the proof of  Proposition \ref{Prp:NormEstimate} and show that
\begin{align*}
\|G(f)\|_{H^{2+\alpha,1+\alpha/2}(Q)}\leq C(1+\|f\|_{H^\alpha}^{1+\alpha/2}).
\end{align*}
Also, if $\Phi:\R\to\R$ is a regular link function in the sense of Definition \ref{Def:LinkFunction} then for each integer $m\geq0$ there exists $C>0$ such that, for all $F\in H^m(\Oo)$,
\begin{align}\label{eq:LinkBound2}
\|\Phi \circ F\|_{H^m}\leq C(1+\|F\|_{H^m}^m).
\end{align}
See \cite[Lemma 29]{Nickl2020} for proof. The result then follows directly from the proof of Theorem \ref{Thm:Contractionf}.

\end{proof}

\subsection{Proofs of the Propositions}\label{SubSec:PropProofs}
We start by proving some useful properties of the solutions to the inhomogeneous problem \eqref{eq:InhomogeneousSch} below.
Let $f\in C(\overline{\Oo})$ and $f>0$. We denote by $S_f$ the forward operator 
\begin{align*}
S_f:H^{2,1}_{B,0}(Q)\to L^2(Q), \quad S_f(u)=\partial_t u-\frac{1}{2}\Delta_xu+fu,
\end{align*}
where 
\begin{align*}
H^{2,1}_{B,0}(Q)=\{u\in H^{2,1}(Q)\g  \text{$u=0$ on $\Sigma$ and $u(x,0)=0$}\}. 
\end{align*}
Then $S_f$ is an isomorphism with a linear continuous inverse operator  
\begin{align*}
V_f:L^2(Q)\to H^{2,1}_{B,0}(Q), \quad h\mapsto V_f(h).
\end{align*}
That is, for any $h\in L^2(\Oo)$ the inhomogeneous equation 
\begin{align}\label{eq:InhomogeneousSch}
\begin{cases}
\partial_t u-\frac{1}{2}\Delta_xu+fu=h \quad &  \text{on $Q$}\\
u= 0 & \text{on $\Sigma$}\\
u(\cdot,0)=0 & \text{on $\Oo$}
\end{cases} 
\end{align}
has a unique weak solution $w_{f,h}=V_f(h)\in H^{2,1}_{B,0}(Q)$, \cite[Chapter 4, Remark 15.1.]{Lions1972b}. 

\begin{lemma}\label{Lem:ConstantBound}
There exists a constant $C>0$, such that for all $f\in C(\overline{\Oo})$, with $f>0$, and $h:[0,\textbf{T}]\times\overline{\Oo}\to\R$ a continuous function with $t\mapsto h(t,\cdot)\in C([0,\textbf{T}];C(\overline{\Oo}))$ the solution $w_{f,h}$ to \eqref{eq:InhomogeneousSch} satisfies
\begin{align}\label{eq:ConstantBound}
\|w_{f,h}\|_{L^2(Q)}\leq C\|h\|_{L^2(Q)}.
\end{align}
\end{lemma}

\begin{proof}
The solution $w_{f,h}$ to \eqref{eq:InhomogeneousSch} has a presentation  
\begin{align*}
w_{f,h}(x,t)=\int_{0}^{t} e^{(t-s)S_f} h(\cdot,s) d s(x), \quad 0 \leq t \leq \textbf{T}, \ x \in \bar{\Omega},  
\end{align*}
see \cite[Theorem 5.1.11]{Lunardi2012}. Using H\"older's inequality we can then write 
\begin{align*}
\|w_{f,h}\|_{L^2}^2 
& =\int_Q\bigg(\int_{0}^{t} e^{(t-s)S_f} h(\cdot,s) d s(x)\bigg)^2d(x,t)\\
& \leq \int_Q\int_{0}^{t} e^{2(t-s)S_f}d s(x)\int_{0}^{t} h^2(x,s) dsd(x,t) \\
& = \int_Q\frac{1}{2}\int_{0}^{2t} e^{(2t-\tilde{s})S_f}d \tilde{s}(x)\int_{0}^{t} h^2( x,s)dsd(x,t) \\
& \leq \int_\Oo\int_{0}^{\textbf{T}} \frac{1}{2}|w_{f,1}(x,2t)|dt\int_{0}^{\textbf{T}} h^2(x,s) dsdx\\
& \leq \textbf{T}\|w_{f,1}\|_\infty\|h\|_{L^2}^2. 
\end{align*}
The solution to \eqref{eq:InhomogeneousSch} has also a probabilistic representation in terms of the Feynman-Kac formula 
\begin{align*}
w_{f,h}(x,t)=\E^x\Bigg(\int_0^{\tau_t}h(X_s,t-s)e^{-\int_0^s f(X_r)dr}ds\Bigg),
\end{align*}
where $(X_s:s\geq0)$ is a $d$-dimensional Brownian motion started at $x\in\Oo$, with exit time $\tau_\Oo$ satisfying $\sup_{x\in\Oo}\E^x(\tau_\Oo)<\infty$, and $\tau_t=\min\{\tau_\Oo,t\}$.
Hence we get a bound  
\begin{align*}
\|w_{f,1}\|_\infty=\sup_{(x,t)\in Q}\Bigg|\E^x\Bigg(\int_0^{\tau_t}e^{-\int_0^s f(X_r)dr}ds\Bigg)\Bigg|\leq \sup_{(x,t)\in Q}\E^x(\tau_t)\leq \textbf{T}. 
\end{align*}  
\end{proof}

Lemma \ref{Lem:ConstantBound} can be used to prove the following stronger regularity estimates. 

\begin{lemma}\label{Lem:H21Estimates}
Let $f,h$ be as in Lemma \ref{Lem:ConstantBound}, and $w_{f,h}\in H^{2,1}_{B,0}(Q)$ be the unique solution to \eqref{eq:InhomogeneousSch}.
Then there exists a constant $C>0$ such that 
\begin{align}
\|w_{f,h}\|_{H^{2,1}(Q)} & \leq C(1+\|f\|_\infty)\|h\|_{L^2(Q)}\label{eq:H21Estimate}\\
\|w_{f,h}\|_{L^2(Q)} & \leq C(1+\|f\|_\infty)\|h\|_{(H^{2,1}_{C,0}(Q))^*},
\end{align}
where 
\begin{align*}
H^{2,1}_{C,0}(Q)=\{u\in H^{2,1}(Q)\g  \text{$u=0$ on $\Sigma$ and $u(x,\textbf{T})=0$}\}. 
\end{align*}
\end{lemma}

\begin{proof}
Using the fact that $S_0:H^{2,1}_{B,0}(Q)\to L^2(Q)$ is an isomorphism and Lemma \ref{Lem:ConstantBound} we get 
\begin{align*}
\|w_{f,h}\|_{H^{2,1}} 
& \leq C\Big\|\Big(\partial_t-\frac{1}{2}\Delta_x\Big)w_{f,h}\Big\|_{L^2}\\
& \leq C(\|S_f(w_{f,h})\|_{L^2}+\|fw_{f,h}\|_{L^2})\\
& \leq C(\|h\|_{L^2}+\|f\|_\infty\|w_{f,h}\|_{L^2})\\
& \leq C(1+\textbf{T}\|f\|_\infty)\|h\|_{L^2}
\end{align*}
which proves the first estimate. 

Denote by $S^*_f=-\partial_t-\frac{1}{2}\Delta_x+f$ the adjoint of $S_f$ and by $H^{2r,r}_{0,0}(Q)$ the closure of $C^\infty_c(Q)$ in $H^{2r,r}(Q)$.
Let $\varphi\in H^{2(r-1), r-1}_{0,0}(Q)$, with some $r\geq1$. Then the adjoint problem 
\begin{align}\label{Eq:Adjoint}
\begin{cases}
S^*_f(u)=\varphi \quad &  \text{on $Q$}\\
u= 0 & \text{on $\Sigma$}\\
u(\cdot,\textbf{T})=0 & \text{on $\Oo$}
\end{cases} 
\end{align}
has a solution $v^*_{f,\varphi}\in X^r(Q)$, where  
\begin{align*}
X^r(Q)=\{u\in H^{2r,r}(Q)\ :\ \text{$u=0$ on $\Sigma$, $u(x,\textbf{T})=0$ and $S^*_f(u)\in H^{2(r-1), r-1}_{0,0}(Q)$}\}.
\end{align*}
The adjoint operator $S^*_f$ is an isomorphism of $X^r(Q)$ onto $H^{2(r-1), r-1}_{0,0}(Q)$, $r\geq1$, and we denote the inverse operator by $V^*_f:H^{2(r-1), r-1}_{0,0}(Q)\to X^r(Q)$, 
\cite[Section 7]{Lions1972b}. 

We start by showing that the estimate \eqref{eq:ConstantBound} holds also for the solution to the adjoint problem \eqref{Eq:Adjoint}. Let $\varphi\in C^\infty_c(Q)$ and $\varphi\not=0$. Then $v^*_{f,\varphi}\not=0$, and we get 
\begin{align*}
\|v^*_{f,\varphi}\|_{L^2}^2
& = \int v^*_{f,\varphi}V^*_f(\varphi)\\
& \leq \|V_f(v^*_{f,\varphi})\|_{L^2}\|\varphi\|_{L^2}\\
& = \|w_{f,v^*_{f,\varphi}}\|_{L^2}\|\varphi\|_{L^2}\\
& \leq \textbf{T}\|v^*_{f,\varphi}\|_{L^2}\|\varphi\|_{L^2},
\end{align*}
where the last inequality follows from Lemma \ref{Lem:ConstantBound}. From the above we see that $\|v^*_{f,\varphi}\|_{L^2}\leq \textbf{T}\|\varphi\|_{L^2}$. 

We can then show, with a similar proof to that of \eqref{eq:H21Estimate}, that 
\begin{align*}
\|v^*_{f,\varphi}\|_{H^{2,1}}
\leq (1+\textbf{T}\|f\|_\infty)\|\varphi\|_{L^2},
\end{align*}
and conclude 
\begin{align*}
\|w_{f,h}\|_{L^2}
& =\sup_{\varphi\in C^\infty_c,\|\varphi\|_{L^2}\leq1}\Big|\int_Q w_{f,h}\varphi\Big|\\
& =\sup_{\varphi\in C^\infty_c,\|\varphi\|_{L^2}\leq1}\Big|\int_Q w_{f,h}S_f^*(V^*_f(\varphi))\Big|\\
& =\sup_{\varphi\in C^\infty_c,\|\varphi\|_{L^2}\leq1}\Big|\int_Q S_f(w_{f,h}) V^*_f(\varphi)\Big|\\
&\leq \sup_{\varphi\in C^\infty_c,\|\varphi\|_{L^2}\leq1} \|v^*_{f,\varphi}\|_{H^{2,1}}\|h\|_{(H^{2,1}_{C,0})^*}\\
& \leq  (1+\textbf{T}\|f\|_\infty)\|h\|_{(H^{2,1}_{C,0})^*}.
\end{align*}
\end{proof}

We now turn to the properties of the forward map $G$. The following norm estimate for the $\mathcal{C}^{2,1}$-H\"older--Zygmund norm of $G(f)=u_f$ is needed for the proof of the Proposition \ref{Prop:ForwardEstimate}.

\begin{lemma}\label{Lem:C21Estimate}
Let $g,u_0$ and $f$ be as in Section \ref{SubSec:BayesApp}. Then for $u_f$, the unique solution to \eqref{eq:SchrodingerHomogeneous}, there exists a constant $C>0$ such that 
\begin{align*} 
\|u_f\|_{\Cc^{2,1}(Q)} \leq C\big(1+\|f\|_\infty\big)\big(\|g\|_{\Cc^{2,1}(\Sigma)}+\|u_0\|_{\Cc^2(\Oo)}\big)
\end{align*}
\end{lemma}

\begin{proof}
We start by noticing that since $f>0$
\begin{align}\label{eq:SupEstimate}
\begin{split}
\|u_f\|_\infty & =\sup_{(x,t)\in Q}\bigg| \E^x\left(u_0(X_t)\chi_{\tau_t=t}\exp\left(-\int_0^tf(X_s)ds\right)\right)\\
&\quad\quad\quad+\E^x\left(g(X_{\tau_\Oo},\tau_\Oo)\chi_{\tau_t<t}\exp\left(-\int_0^{\tau_\Oo} f(X_s)ds\right)\right)\bigg|\\
& \leq \|u_0\|_\infty+\|g\|_\infty. 
\end{split}
\end{align}  
We will also need  the isomorphism, see e.g. \cite{Amann2010},
\begin{align*}
\Big(\partial_t-\frac{1}{2}\Delta_x,\text{tr}_{|\Sigma},\text{tr}_{|\Oo}\Big)
& :\Cc^{2,1}(Q) \to \mathcal{X}\\
u & \mapsto \Big(\partial_tu-\frac{1}{2}\Delta_xu,\text{tr}_{|\Sigma}(u),u(\cdot,0)\Big),
\end{align*}
where $\mathcal{X}$ is the subspace of $C(Q)\times\Cc^{2,1}(\Sigma)\times\Cc^{2}(\Oo)$ of the elements $(h,g,u_0)$ satisfying the consistency conditions 
\begin{align*}
g(x,0)=u_0(x)
\quad\text{and} \quad
\partial_t g(x,0)-\frac{1}{2}\Delta_x u_0(x)=h(x,0),  
\quad x\in\partial\Oo.
\end{align*}
Using the above we get 
\begin{align*}  
\|u_f\|_{\Cc^{2,1}(Q)}
& \leq C\bigg(\Big\|\Big(\partial_t-\frac{1}{2}\Delta_x\Big)u_f\Big\|_{\infty}+\|g\|_{\Cc^{2,1}(\Sigma)}+\|u_0\|_{\Cc^2(\Oo)}\bigg)\\
& = C\big(\|fu_f\|_{\infty}+\|g\|_{\Cc^{2,1}(\Sigma)}+\|u_0\|_{\Cc^2(\Oo)}\big)\\
& \leq C\big(\|f\|_\infty \|u_f\|_{\infty}+\|g\|_{\Cc^{2,1}(\Sigma)}+\|u_0\|_{\Cc^2(\Oo)}\big)\\
& \leq C\big(1+\|f\|_\infty\big)\big(\|g\|_{\Cc^{2,1}(\Sigma)}+\|u_0\|_{\Cc^2(\Oo)}\big)\\
\end{align*}
\end{proof}

Using the above Lemmas we can show that the forward operator $G$ satisfies the following Lipschitz condition. 

\begin{lemma}\label{Lem:ForwardEstimate}
Let $g,u_0$ and $f$ be as in Lemma \ref{eq:SupEstimate}. Then, for the unique solution $u_f$ to \eqref{eq:SchrodingerHomogeneous}, there exists a constant $C>0$ such that 
\begin{align*}
\|u_{f_1}-u_{f_2}\|_{L^2(Q)} \leq C(1+\|f_1\|_\infty)(1+\|f_2\|_\infty)\|f_1-f_2\|_{(H^2_0(\Oo))^*}.
\end{align*}
\end{lemma}

\begin{proof}
Let $u_{f_i}$, $i=1,2$, be solutions to \eqref{eq:SchrodingerHomogeneous}. We notice that $w=u_{f_1}-u_{f_2}$ solves the inhomogeneous equation $S_{f_1}(w)=(f_1-f_2)u_{f_2}$ on $Q$ and $w=0$ on the boundary $\Sigma\times\overline{\Oo}$. Using Lemmas  \ref{Lem:H21Estimates} and \ref{Lem:C21Estimate} we see that 
\begin{align*}
\|u_{f_1}-u_{f_2}\|_{L^2} 
& = \|w_{f_1,(f_1-f_2)u_{f_2}}\|_{L^2}\\
& \leq C(1+\textbf{T}\|f_1\|_\infty)\|(f_1-f_2)u_{f_2}\|_{(H^{2,1}_{C,0})^*}\\
& \leq C(1+\textbf{T}\|f_1\|_\infty)\sup_{\varphi\in H^{2,1}_{C,0},\|\varphi\|_{H^{2,1}}\leq1}
\Big|\int_Q(f_1-f_2)u_{f_2}\varphi\Big|\\
& \leq C(1+\textbf{T}\|f_1\|_\infty)\sup_{\varphi\in H^{2,1}_{C,0},\|\varphi\|_{H^{2,1}}\leq1}
\|u_{f_2}\varphi\|_{H^{2,1}}\|f_1-f_2\|_{(H^{2,1}_{C,0})^*}\\
& \leq C(1+\textbf{T}\|f_1\|_\infty)\|u_{f_2}\|_{\Cc^{2,1}}\|f_1-f_2\|_{(H^{2,1}_{C,0})^*}\\
& \leq C(1+\|f_1\|_\infty)(1+\|f_2\|_\infty)\|f_1-f_2\|_{(H^{2,1}_{C,0})^*}.
\end{align*}
Furthermore
\begin{align*}
\|f_1-f_2\|_{(H^{2,1}_{C,0})^*} 
& = \sup_{\varphi\in H^{2,1}_{C,0},\|\varphi\|_{H^{2,1}}\leq1} \Big|\int_Q(f_1(x)-f_2(x))\varphi(x,t)d(x,t)\Big|\\
& = \sup_{\varphi\in H^{2,1}_{C,0},\|\varphi\|_{H^{2,1}}\leq1} \Big\|\int_0^\textbf{T} \varphi(\cdot,t)dt\Big\|_{H^2}\|f_1-f_2\|_{(H^2_0)^*}, 
\end{align*}
where 
\begin{align*}
\Big\|\int_0^\textbf{T} \varphi(\cdot,t)dt\Big\|_{H^2(\Oo)}^2   
& = \sum_{|\alpha|\leq2} \Big\|D^\alpha_x\int_0^\textbf{T} \varphi(\cdot,t)dt\Big\|_{L^2(\Oo)}^2\\
& = \sum_{|\alpha|\leq2} \int_\Oo\Big|\int_0^\textbf{T}  D^\alpha_x\varphi(\cdot,t)dt\Big|^2dx\\
& \leq \sum_{|\alpha|\leq2} \int_\Oo \textbf{T}\int_0^\textbf{T}  |D^\alpha_x\varphi(\cdot,t)|^2dtdx\\
& \leq \textbf{T} \int_0^\textbf{T}\|\varphi(\cdot,t)\|_{H^2}^2dt\\
& \leq \textbf{T}\|\varphi(\cdot,t)\|_{H^{2,1}}^2
\end{align*}
which concludes the proof. 
\end{proof}

We can finally proceed to prove Propositions \ref{Prop:ForwardEstimate}-\ref{Prp:NormEstimate}.

\begin{proof}[Proof of Proposition \ref{Prop:ForwardEstimate}]
Note that if $\Phi:\R\to\R$ is a regular link function in the sence of Definition \ref{Def:LinkFunction} then there exists $C>0$ such that for all $F\in L^\infty(\Oo)$
\begin{align*}
\|\Phi \circ F\|_\infty\leq C(1+\|F\|_\infty).
\end{align*}
Also, for all $F_1,F_2\in C^2(\Oo)$  there exists $C>0$ such that 
\begin{align*}
\|\Phi\circ F_1-\Phi \circ F_2\|_{(H^2)^*}\leq C(1+\|F_1\|_{C^2}^2+\|F_2\|_{C^2}^2)\|F_1-F_2\|_{(H^2)^*}. 
\end{align*}
See \cite[Lemma 29]{Nickl2020} for a proof.     
Using Lemma \ref{Lem:ForwardEstimate} and the above estimates we can then write 
\begin{align*}
\|\G(F_1)-\G(F_2)\|_{L^2(Q)}
& \leq C(1+\|f_1\|_\infty)(1+\|f_2\|_\infty)\|f_1-f_2\|_{(H^2_0(\Oo))^*}\\
& \leq C(1+\|F_1\|^2_\infty\vee\|F_2\|_\infty^2)(1+\|F_1\|^2_{C^2}\vee\|F_2\|_{C^2}^2)\|F_1-F_2\|_{(H^2(\Oo))^*}\\
& \leq C(1+\|F_1\|_{C^2(\Oo)}^4\vee\|F_2\|_{C^2(\Oo)}^4)\|F_1-F_2\|_{(H^2(\Oo))^*}\\
\end{align*}
\end{proof}

\begin{proof}[Proof of Proposition \ref{Prp:StabilityEstimate}]
Applying Jensen's inequality we see that 
\begin{align*}
\inf_{(x,t)\in Q}u_f(x,t) & =\inf_{(x,t)\in Q}\bigg( \E^x\left(u_0(X_t)\chi_{\tau_t=t}e^{-\int_0^tf(X_s)ds}\right)\\
&\quad\quad\quad+\E^x\left(g(X_{\tau_\Oo},\tau_\Oo)\chi_{\tau_t<t}e^{-\int_0^{\tau_{\Oo}} f(X_s)ds}\right)\bigg)\\
& \geq u_{0,\min}e^{-\textbf{T}\|f\|_\infty}+g_{\min}\inf_{x\in\Oo}e^{-\|f\|_\infty\E^x(\tau_\Oo)}\\
& \geq (u_{0,\min}+g_{\min})e^{-C_\textbf{T}\|f\|_\infty}>0
\end{align*}
where $C_\textbf{T}=\max\{\textbf{T},\E^x(\tau_\Oo)\}$.

Since $f\geq f_{\min}>0$ the solution $u_f$ is positive and we can write $f(x,t)=\frac{(\frac{1}{2}\Delta_x-\partial_t)u_f(x,t)}{u_f(x,t)}$. Note that $f(x,t)$ is a constant in $t$. We can then write 
\begin{align*}
\|f_1-f_2\|_{L^2(Q)}
& = \Big\|\frac{(\frac{1}{2}\Delta_x-\partial_t)u_{f_1}}{u_{f_1}}-\frac{(\frac{1}{2}\Delta_x-\partial_t)u_{f_2}}{u_{f_2}}\Big\|_{L^2(Q)}\\
& \leq \Big\|\frac{(\frac{1}{2}\Delta_x-\partial_t)(u_{f_1}-u_{f_2})}{u_{f_1}}\Big\|_{L^2(Q)}
+ \Big\|(u_{f_1}^{-1}-u_{f_2}^{-1})\Big(\frac{1}{2}\Delta_x-\partial_t\Big)u_{f_2}\Big\|_{L^2}\\
& \lesssim (\inf_{(x,t)\in Q}|u_{f_1}(x,t)|)^{-1}\|u_{f_1}-u_{f_2}\|_{H^{2,1}(Q)}
+ \|u_{f_1}^{-1}-u_{f_2}^{-1}\|_{L^2(Q)}\|f_2u_{f_2}\|_{\Cc^0}. 
\end{align*}
To bound the last term we note that 
\begin{align*}
\Big|\frac{1}{u_{f_1}}-\frac{1}{u_{f_2}}\Big|
\leq \Big|\frac{u_{f_1}-u_{f_2}}{\min\{u_{f_1}^2,u_{f_2}^2\}}\Big|
\leq (u_{0,\min}+g_{\min})^{-2}e^{2C_\textbf{T}\|f_1\vee f_2\|_\infty}|u_{f_1}-u_{f_2}|.
\end{align*}
Combining the above with \eqref{eq:SupEstimate} we get
\begin{align*}
\|f_1-f_2\|_{L^2}
& \leq (u_{0,\min}+g_{\min})^{-1} e^{C_\textbf{T}\|f_1\|_\infty}\|u_{f_1}-u_{f_2}\|_{H^{2,1}} \\
& + \|f_2\|_\infty(\|u_0\|_\infty+\|g\|_\infty)
(u_{0,\min}+g_{\min})^{-2}e^{C_\textbf{T}\|f_1\vee f_2\|_\infty}\|u_{f_1}-u_{f_2}\|_{L^2}
\end{align*}   
We conclude the proof by noting that $\|f_1-f_2\|_{L^2(Q)}=\textbf{T}\|f_1-f_2\|_{L^2(\Oo)}$.
\end{proof}

\begin{proof}[Proof of Proposition \ref{Prp:NormEstimate}]
We start by noting that since $f\in C^\beta\subset H^\beta$ and the assumption on $g,u_0$ the solution $u_f=G(f)\in H^{2+\beta,1+\beta/2}$, see \cite[Theorem 5.3]{Lions1972b}.     
Denote by 
\begin{align*}
\F_{(C.C.)}^\beta\subset \F^\beta= H^{\beta,\beta/2}(Q)\times H^{3/2+\beta,3/4+\beta/2}(\Sigma)\times H^{1+\beta}(\Oo)
\end{align*}
the subspace of $\F^\beta$ of elements $(h,g,u_0)$ that satisfy the following compatibility conditions $(C.C.)$; There exists $\psi\in H^{2+\beta,1+\beta/2}(Q)$ such that 
\begin{align*}
\text{$\psi=g$ on $\Sigma$,} &  \text{$\quad \psi(x,0)=u_0$ on $\Oo$ and}\\
\quad \partial_t^k ((\partial_t-\frac{1}{2}\Delta_x+ f) & \psi)|_{t=0} =\partial_t^k h(x,0) \text{\ \ for $0\leq k < \frac{\beta}{2}-\frac{1}{2}$}.
\end{align*}
Using the isomorphism \cite[Theorem 6.2]{Lions1972b}
\begin{align*}
\Big(\partial_t-\frac{1}{2}\Delta_x,\text{tr}_{|\Sigma},\text{tr}_{|\Oo}\Big)
& : H^{2+\beta,1+\beta/2}\to \F_{(C.R)}^\beta\\
u & \mapsto \Big(\partial_tu-\frac{1}{2}\Delta_xu,\text{tr}_{|\Sigma}(u),u(\cdot,0)\Big),
\end{align*}  
inequality \eqref{eq:SobolevBound2}, and the interpolation inequality \eqref{eq:Interpolation} we get 
\begin{align*}
\|u_f\|_{H^{2+\beta, 1+\beta/2}}
& \lesssim \|(\partial_t-\frac{1}{2}\Delta_x)u_f\|_{H^{\beta,\beta/2}}
+\|tr_{|_\Sigma}(u_f)\|_{H^{3/2+\beta,3/4+\beta/2}}
+\|u(\cdot,0)\|_{H^{1+\beta}}\\
& \lesssim \|fu_f\|_{H^{\beta,\beta/2}}    
+\|g\|_{C^{2+\beta,1+\beta/2}}
+\|u_0\|_{C^{2+\beta}}\\
& \lesssim 1+  \|f\|_{\Cc^{\beta}}\|u_f\|_{H^{\beta,\beta/2}}\\
& \lesssim 1+  \|f\|_{\Cc^{\beta}}
\|u_f\|_{H^{2+\beta,1+\beta/2}}^{\frac{\beta}{2+\beta}}\|u_f\|_{L^2}^\frac{2}{2+\beta}. 
\end{align*}
If $\|u_f\|_{H^{2+\beta,1+\beta/2}}\geq1$ we can divide both sides by $\|u_f\|_{H^{2+\beta,1+\beta/2}}^{\frac{\beta}{2+\beta}}$ and otherwise estimate the norm on the right hand side by $1$. Using the second part of Proposition \ref{Prop:ForwardEstimate} we then see that 
\begin{align*}
\|u_f\|_{H^{2+\beta, 1+\beta/2}} 
\lesssim 1+  \|f\|_{\Cc^{\beta}}^{1+\beta/2}\|u_f\|_{L^2}
\lesssim 1+  \|f\|_{\Cc^{\beta}}^{1+\beta/2}\|u_f\|_{\infty}
\lesssim 1+  \|f\|_{\Cc^{\beta}}^{1+\beta/2}.
\end{align*}
\end{proof}

\textbf{Acknowledgments.} The author would like to thank Richard Nickl for valuable discussions.

\bibliographystyle{siam}
\bibliography{Inverse_problems_references} 

\begin{thebibliography}{10}

\bibitem{Abraham2020}
{\sc K.~Abraham and R.~Nickl}, {\em On statistical calder{\'o}n problems},
  Mathematical Statistics and Learning, 2 (2020), pp.~165--216.

\bibitem{Agapiou2013}
{\sc S.~Agapiou, S.~Larsson, and A.~M. Stuart}, {\em Posterior contraction
  rates for the {B}ayesian approach to linear ill-posed inverse problems},
  Stochastic Processes and their Applications, 123 (2013), pp.~3828--3860.

\bibitem{Altmeyer2019}
{\sc R.~Altmeyer and M.~Rei{\ss}}, {\em Nonparametric estimation for linear
  spdes from local measurements}, arXiv preprint arXiv:1903.06984,  (2019).

\bibitem{Amann1995}
{\sc H.~Amann}, {\em Linear and quasilinear parabolic problems}, vol.~II:
  Function spaces, Springer, Basel: Birkhäuser, 1995.

\bibitem{Amann2010}
\leavevmode\vrule height 2pt depth -1.6pt width 23pt, {\em Anisotropic function
  spaces and maximal regularity for parabolic problems. Part 1}, vol.~6,
  Matfyzpress, 2010.

\bibitem{Arridge2019}
{\sc S.~Arridge, P.~Maass, O.~{\"O}ktem, and C.-B. Sch{\"o}nlieb}, {\em Solving
  inverse problems using data-driven models}, Acta Numerica, 28 (2019),
  pp.~1--174.

\bibitem{Bachmayr2009}
{\sc M.~Bachmayr and M.~Burger}, {\em Iterative total variation schemes for
  nonlinear inverse problems}, Inverse Problems, 25 (2009), p.~105004.

\bibitem{Bal2012}
{\sc G.~Bal}, {\em Hybrid inverse problems and internal functionals}, Inverse
  problems and applications: inside out. II, 60 (2012), pp.~325--368.

\bibitem{Bal2010}
{\sc G.~Bal and G.~Uhlmann}, {\em Inverse diffusion theory of photoacoustics},
  Inverse Problems, 26 (2010), p.~085010.

\bibitem{Benning2017}
{\sc M.~Benning, M.~M. Betcke, M.~J. Ehrhardt, and C.-B. Sch{\"o}nlieb}, {\em
  Choose your path wisely: gradient descent in a {B}regman distance framework},
  arXiv preprint arXiv:1712.04045,  (2017).

\bibitem{Benning2018}
{\sc M.~Benning and M.~Burger}, {\em Modern regularization methods for inverse
  problems}, Acta Numerica, 27 (2018), p.~1–111.

\bibitem{Beskos2017}
{\sc A.~Beskos, M.~Girolami, S.~Lan, P.~E. Farrell, and A.~M. Stuart}, {\em
  Geometric {MCMC} for infinite-dimensional inverse problems}, Journal of
  Computational Physics, 335 (2017), pp.~327--351.

\bibitem{Bissantz2007}
{\sc N.~Bissantz, T.~Hohage, A.~Munk, and F.~Ruymgaart}, {\em Convergence rates
  of general regularization methods for statistical inverse problems and
  applications}, SIAM Journal on Numerical Analysis, 45 (2007), pp.~2610--2636.

\bibitem{Calderon1980}
{\sc A.~P. Calder{\'o}n}, {\em On an inverse boundary value problem}, Seminar
  on Numerical Analysis and its Applications to Continuum Physics (Rio de
  Janeiro, 1980),  (1980), pp.~65–--73.

\bibitem{Colton1998}
{\sc D.~Colton and R.~Kress}, {\em Inverse acoustic and electromagnetic
  scattering theory}, Springer-Verlag, Berlin, second~ed., 1998.

\bibitem{Dashti2017}
{\sc M.~Dashti and A.~M. Stuart}, {\em The Bayesian Approach to Inverse
  Problems. In: Handbook of Uncertainty Quantification}, Springer International
  Publishing, Cham, 2017, pp.~311--428.

\bibitem{Daubechies2006}
{\sc I.~Daubechies}, {\em Ten lectures on wavelets (Ninth printing, 2006)},
  vol.~61 of {C}{B}{M}{S}-{N}{S}{F} Regional conference series in applied
  mathematics, SIAM, 2006.

\bibitem{Dosso2012}
{\sc S.~E. Dosso, C.~W. Holland, and M.~Sambridge}, {\em Parallel tempering for
  strongly nonlinear geoacoustic inversion}, The Journal of the Acoustical
  Society of America, 132 (2012), pp.~3030--3040.

\bibitem{Engl1996a}
{\sc H.~W. Engl, M.~Hanke, and A.~Neubauer}, {\em Regularization of inverse
  problems}, vol.~375 of Mathematics and its Applications, Kluwer Academic
  Publishers Group, Dordrecht, 1996.

\bibitem{Freidlin1985}
{\sc M.~I. Freidlin}, {\em Functional integration and partial differential
  equations}, no.~109, Princeton university press, 1985.

\bibitem{Friedman1964}
{\sc A.~Friedman}, {\em Partial differential equations of parabolic type},
  Prentice--{H}all, 1964.

\bibitem{Ghosal2017}
{\sc S.~Ghosal and A.~van~der Vaart}, {\em Fundamentals of nonparametric
  {B}ayesian inference}, vol.~44, Cambridge University Press, 2017.

\bibitem{Gine2015}
{\sc E.~Gin{\'e} and R.~Nickl}, {\em Mathematical foundations of
  infinite-dimensional statistical models}, vol.~40, Cambridge University
  Press, 2015.

\bibitem{Giordano2020a}
{\sc M.~Giordano and H.~Kekkonen}, {\em Bernstein--von {M}ises theorems and
  uncertainty quantification for linear inverse problems}, SIAM/ASA Journal on
  Uncertainty Quantification, 8 (2020), pp.~342--373.

\bibitem{Giordano2020b}
{\sc M.~Giordano and R.~Nickl}, {\em Consistency of {B}ayesian inference with
  {G}aussian process priors in an elliptic inverse problem}, Inverse Problems,
  (2020).

\bibitem{Hairer2014}
{\sc M.~Hairer, A.~M. Stuart, and S.~J. Vollmer}, {\em Spectral gaps for a
  {M}etropolis--{H}astings algorithm in infinite dimensions}, Annals of Applied
  Probability, 24 (2014), pp.~2455--2490.

\bibitem{Hildebrandt2019}
{\sc F.~Hildebrandt and M.~Trabs}, {\em Parameter estimation for {SPDE}s based
  on discrete observations in time and space}, arXiv preprint arXiv:1910.01004,
   (2019).

\bibitem{Hohage2015}
{\sc T.~Hohage and F.~Weidling}, {\em Verification of a variational source
  condition for acoustic inverse medium scattering problems}, Inverse Problems,
  31 (2015), p.~075006.

\bibitem{Isaacson2004}
{\sc D.~Isaacson, J.~L. Mueller, J.~C. Newell, and S.~Siltanen}, {\em
  Reconstructions of chest phantoms by the {D}-bar method for electrical
  impedance tomography}, IEEE Transactions on Medical Imaging, 23 (2004),
  pp.~821--828.

\bibitem{Kaipio2004a}
{\sc J.~Kaipio and E.~Somersalo}, {\em Statistical and computational inverse
  problems}, vol.~160 of Applied Mathematical Sciences, Springer-Verlag, New
  York, 2005.

\bibitem{Kaltenbacher2008}
{\sc B.~Kaltenbacher, A.~Neubauer, and O.~Scherzer}, {\em Iterative
  regularization methods for nonlinear ill-posed problems}, vol.~6, de Gruyter,
  2008.

\bibitem{Kekkonen2016}
{\sc H.~Kekkonen, M.~Lassas, and S.~Siltanen}, {\em Posterior consistency and
  convergence rates for {B}ayesian inversion with hypoelliptic operators},
  Inverse Problems, 32 (2016), p.~085005.

\bibitem{Knapik2018}
{\sc B.~Knapik and J.-B. Salomond}, {\em A general approach to posterior
  contraction in nonparametric inverse problems}, Bernoulli, 24 (2018),
  pp.~2091--2121.

\bibitem{Knapik2011}
{\sc B.~T. Knapik, A.~W. van~der Vaart, and J.~H. van Zanten}, {\em Bayesian
  inverse problems with {G}aussian priors}, The Annals of Statistics, 39
  (2011), pp.~2626--2657.

\bibitem{Krylov1996}
{\sc N.~V. Krylov}, {\em Lectures on elliptic and parabolic equations in
  {H}\"older spaces}, no.~12, American Mathematical Soc., 1996.

\bibitem{Kuchment2015}
{\sc P.~Kuchment and L.~Kunyansky}, {\em Mathematics of photoacoustic and
  thermoacoustic tomography}, in Handbook of Mathematical Methods in Imaging,
  (2015), p.~1117–1167.

\bibitem{Latz2020}
{\sc J.~Latz, J.~P. Madrigal-Cianci, F.~Nobile, and R.~Tempone}, {\em
  Generalized parallel tempering on bayesian inverse problems}, arXiv preprint
  arXiv:2003.03341,  (2020).

\bibitem{Lions1972}
{\sc J.-L. Lions and E.~Magenes}, {\em Non-homogeneous boundary value problems
  and applications, Vol. 1}, Springer--Verlag, New York--Heidelberg, 1972.

\bibitem{Lions1972b}
\leavevmode\vrule height 2pt depth -1.6pt width 23pt, {\em Non-homogeneous
  boundary value problems and applications, Vol. 2}, Springer--Verlag, New
  York--Heidelberg, 1972.

\bibitem{Lunardi2012}
{\sc A.~Lunardi}, {\em Analytic semigroups and optimal regularity in parabolic
  problems}, Springer Science \& Business Media, 2012.

\bibitem{Ma2015}
{\sc Y.-A. Ma, T.~Chen, and E.~B. Fox}, {\em A complete recipe for stochastic
  gradient {MCMC}}, in Proceedings of the 28th International Conference on
  Neural Information Processing Systems-Volume 2, 2015, pp.~2917--2925.

\bibitem{Meyer1995}
{\sc Y.~Meyer}, {\em Wavelets and operators}, vol.~1, Cambridge University
  Press, Cambridge, UK, 1995.

\bibitem{Monard2019}
{\sc F.~Monard, R.~Nickl, and G.~P. Paternain}, {\em Efficient nonparametric
  {B}ayesian inference for {X}-ray transforms}, Annals of Statistics, 47
  (2019), pp.~1113--1147.

\bibitem{Monard2020}
\leavevmode\vrule height 2pt depth -1.6pt width 23pt, {\em Statistical
  guarantees for {B}ayesian uncertainty quantification in non-linear inverse
  problems with gaussian process priors}, arXiv preprint arXiv:2007.15892,
  (2020).

\bibitem{Monard2021}
{\sc F.~Monard, R.~Nickl, and G.~P. Paternain}, {\em Consistent inversion of
  noisy non-{A}belian {X}-ray transforms}, Communications on Pure and Applied
  Mathematics,  (2021).

\bibitem{Nickl2020a}
{\sc R.~Nickl}, {\em Bernstein–-von {M}ises theorems for statistical inverse
  problems {I}: {S}chr\"odinger equation}, J. Eur. Math. Soc., 22 (2020),
  pp.~2697--2750.

\bibitem{Nickl2017}
{\sc R.~Nickl and J.~S{\"o}hl}, {\em Nonparametric {B}ayesian posterior
  contraction rates for discretely observed scalar diffusions}, The Annals of
  Statistics, 45 (2017), pp.~1664--1693.

\bibitem{Nickl2019}
\leavevmode\vrule height 2pt depth -1.6pt width 23pt, {\em Bernstein--von
  {M}ises theorems for statistical inverse problems {II}: compound {P}oisson
  processes}, Electronic Journal of Statistics, 13 (2019), pp.~3513--3571.

\bibitem{Nickl2020}
{\sc R.~Nickl, S.~van~de Geer, and S.~Wang}, {\em Convergence rates for
  penalized least squares estimators in {PDE} constrained regression problems},
  SIAM/ASA Journal on Uncertainty Quantification, 8 (2020), pp.~374--413.

\bibitem{Nickl2020b}
{\sc R.~Nickl and S.~Wang}, {\em On polynomial-time computation of
  high-dimensional posterior measures by {L}angevin-type algorithms}, arXiv
  preprint arXiv:2009.05298,  (2020).

\bibitem{Patriarca2004}
{\sc M.~Patriarca and T.~Lepp{\"a}nen}, {\em Modeling language competition},
  Physica A: Statistical Mechanics and its Applications, 338 (2004),
  pp.~296--299.

\bibitem{Ray2013}
{\sc K.~Ray}, {\em Bayesian inverse problems with non-conjugate priors},
  Electronic Journal of Statistics, 7 (2013), pp.~2516--2549.

\bibitem{Sole2010}
{\sc R.~V. Sol{\'e}, B.~Corominas-Murtra, and J.~Fortuny}, {\em Diversity,
  competition, extinction: the ecophysics of language change}, Journal of The
  Royal Society Interface, 7 (2010), pp.~1647--1664.

\bibitem{Stuart2010}
{\sc A.~M. Stuart}, {\em Inverse problems: a {B}ayesian perspective}, Acta
  Numerica, 19 (2010), pp.~451--559.

\bibitem{Szabo2015}
{\sc B.~Szab{\'o}, A.~van~der Vaart, and J.~van Zanten}, {\em Frequentist
  coverage of adaptive nonparametric {B}ayesian credible sets}, The Annals of
  Statistics, 43 (2015), pp.~1391--1428.

\bibitem{Triebel1983}
{\sc H.~Triebel}, {\em Theory of Function Spaces}, vol.~78 of Monographs in
  Mathematics, Birkh\"auser Verlag, 1983.

\bibitem{Triebel2008}
\leavevmode\vrule height 2pt depth -1.6pt width 23pt, {\em Function spaces and
  wavelets on domains}, vol.~7 of Tracts in Mathematics, European Mathematical
  Society, 2008.

\bibitem{Vollmer2013}
{\sc S.~J. Vollmer}, {\em Posterior consistency for {B}ayesian inverse problems
  through stability and regression results}, Inverse Problems, 29 (2013),
  p.~125011.

\bibitem{Wu2012}
{\sc J.~Wu}, {\em Theory and applications of partial functional differential
  equations}, vol.~119, Springer Science \& Business Media, 2012.

\end{thebibliography}

\end{document}